\documentclass[12pt]{extarticle}
\usepackage{amsmath, amsthm, amssymb, color}
\usepackage[colorlinks=true,linkcolor=blue,urlcolor=blue]{hyperref}
\usepackage{graphicx}
\usepackage{caption}
\usepackage{mathtools}
\usepackage{enumerate}
\usepackage{verbatim}
\usepackage{tikz,tikz-cd,tikz-3dplot}
\usepackage{amssymb}
\usetikzlibrary{matrix}
\usetikzlibrary{arrows}
\usepackage{algorithm}
\usepackage[noend]{algpseudocode}
\usepackage{caption}
\usepackage[normalem]{ulem}
\usepackage{subcaption}
\usepackage{multirow}
\oddsidemargin \evensidemargin
\usepackage{makecell}
\usepackage{array}
\usepackage{enumitem}
\usepackage{xcolor}

\usepackage{dsfont}
\usepackage{comment}

\usepackage{algorithm}
\usepackage{algpseudocode}

\usepackage{hyperref}
\newtheorem{theorem}{Theorem}[section]

\newtheorem{proposition}[theorem]{Proposition}
\newtheorem{lemma}[theorem]{Lemma}
\newtheorem{corollary}[theorem]{Corollary}

\theoremstyle{definition}

\newtheorem{remark}[theorem]{Remark}

\newtheorem{example}[theorem]{Example}

\newcommand{\PP}{\mathbb{P}}
\newcommand{\RR}{\mathbb{R}}

\newcommand{\ZZ}{\mathbb{Z}}
\newcommand{\NN}{\mathbb{N}}

\newcommand{\True}{\textbf{true}}
\newcommand{\False}{\textbf{false}}

\DeclareMathOperator{\Newt}{Newt} 
\DeclareMathOperator{\relint}{relint}
\DeclareMathOperator{\conv}{conv}

\DeclareMathOperator{\interior}{int}

\title{\bf Copositive geometry of Feynman integrals}
\author{Bernd Sturmfels and M\'at\'e L. Telek \bigskip \\
\begin{footnotesize} {\em Dedicated to the memory of Victoria Powers} \end{footnotesize}}

\date{}

\begin{document}

\maketitle

\begin{abstract}
\noindent
Copositive matrices and copositive polynomials 
are objects from optimization. We  connect these
 to the geometry of Feynman integrals in physics.
The integral is guaranteed to converge if its
kinematic parameters lie in the copositive cone. 
P\'olya's method makes this manifest.
We
study the copositive cone for the
second Symanzik polynomial of any Feynman~graph.
Its algebraic boundary is described by Landau discriminants.

\end{abstract}
 
 \section{Introduction}

When evaluating the integral of a rational function, 
convergence is related to positivity properties of its denominator polynomial.
For an example, consider the univariate integral
\begin{equation}
\label{eq:univariateintegral}
\int_{0}^\infty \frac{x (x+1)^2}{(\,x^2 - (2-\epsilon) x + 1\,)^3 } dx.
\end{equation}
This integral converges if and only if 
$\epsilon$ is positive, so the denominator is positive for $x \geq 0$.
Integrals in geometry are often written in homogeneous coordinates.
This equates (\ref{eq:univariateintegral}) with
\begin{equation}
\label{eq:univariateintegral2}\qquad
\int_{\PP^1_{> 0}} \frac{x_1^2 x_2^2 (x_1+x_2)^2}{(\,x_1^2 - (2-\epsilon) x_1 x_2 + x_2^2\,)^3} \, \Omega
\qquad {\rm where} \,\, \,\,\Omega \,= \, \frac{d x_1}{x_1} - \frac{d x_2}{x_2}.
\end{equation}

The title of this paper is a nod to physics.
 Feynman integrals \cite{Weinzierl} are building blocks for
scattering amplitudes \cite{Borinsky, HennRaman, MizeraTelen}. They take the form shown in
(\ref{Eq:FeynmParamInt}) below.
The integral $I_G(z)$ is derived from a 
Feynman graph $G$ with $n$ edges and $\ell$ loops.
The integrand contains the two Symanzik polynomials
$\,\mathcal{U} $ and $\mathcal{F}_z $.
These are polynomials in $x = (x_1,\ldots,x_n)$,
of degrees $\ell$ and $\ell+1$ respectively, 
and $\mathcal{F}_z$ depends linearly on a vector
$z$ of  kinematic parameters.

Copositive geometry is a wordplay
which connects positive geometry \cite{RST} to
copositivity in optimization \cite{Dur, Nie:Book, Vargas}.
A homogeneous polynomial is {\em copositive}
if it is nonnegative on the positive orthant.
A symmetric matrix is copositive if
the associated quadratic form is copositive \cite{Motzkin}.
Our object of study is the {\em copositive cone}
$\mathcal{C}_G$ of a Feynman graph $G$.
This is the set of parameter vectors $z$
for which the second Symanzik polynomial $\mathcal{F}_z$ is copositive.

This project started from our attempt to understand the
{\em Euclidean region} of a graph~$G$.
This refers to a region in $z$-space where the 
Feynman integral $I_G(z)$ converges.
Details vary across the physics literature. We opted for the
definition of Henn and Raman in \cite{HennRaman},
which says that the Euclidean region is precisely the copositive cone $\mathcal{C}_G$.
In  \cite[Section 3.1]{HennRaman}, they write that
``the general question of determining the Euclidean region is
an interesting open question''.
The purpose of this article is to suggest mathematical answers to that question.

We study the copositive cone $\mathcal{C}_G$ in the general
setting when all particles are massive.
This ensures that every variable $x_i$ occurs to the
second power in $\mathcal{F}_z(x)$. It is  that
feature which makes our problem interesting.
If one restricts to massless particles, then 
$\mathcal{C}_G$ becomes a polyhedral cone.
That case is not interesting for us.
In this paper,  the copositive cone $\mathcal{C}_G$ is always nonlinear.
Its boundary is given algebraically by {\em Landau discriminants}~\cite{FMT,MizeraTelen}.
Most of our results remain true in the mixed case, where some particles are massive and others are massless, which is a relevant scenario in physics. However, the proof of our main theorem (Theorem~\ref{Thm:PolyaSymanzik}) requires that all particles be massive.

We present a small example which illustrates the
various ingredients for  our discussion.

\begin{example}[Bubble diagram] \label{ex:bubble}
Consider the following Feynman graph $G$ with $\ell = 1$ loop:
\begin{equation}
\label{Eq:Bubble}
\begin{aligned}
\begin{tikzpicture}[scale=0.3]
    \node[] (P4) at (4.6, 2) {\small $p_4$};
     \node[] (P3) at (4.6, -2) {\small $p_3$};
    \node[] (P1) at (-4.6, 2) {\small $p_1$};
     \node[] (P2) at (-4.6, -2) {\small $p_2$};
    \node[fill=black, circle, inner sep=1.5pt]
    (A) at (2.5, 0) {};
     \node[fill=black, circle, inner sep=1.5pt] (B) at (-2.5, 0) {};
     \node[fill=gray, circle, inner sep=1.5pt] (A1) at (5.5, 1.8) {};
      \node[fill=gray, circle, inner sep=1.5pt]  (A2) at (5.5, -1.8) {};
    \node[fill=gray, circle, inner sep=1.5pt] (B1) at (-5.5, 1.8) {};
      \node[fill=gray, circle, inner sep=1.5pt] (B2) at (-5.5, -1.8) {};
     \draw[bend left=70] (A) to (B) node[midway, above=-45 pt] {\Huge $m_2$};
     \draw[bend right=70] (A) to (B) node[midway, above=45 pt] {\Huge $m_1$};
    \draw (A) -- (A1);
    \draw (A) -- (A2);
       \draw (B) -- (B1);
    \draw (B) -- (B2);
\end{tikzpicture}
\end{aligned}
\end{equation}
This describes a scattering process with four external particles
and $n=2$ internal particles, with masses
$m_1$ and $m_2$. The graph has two spanning trees
and one spanning $2$-forest:
     \begin{center}
\begin{minipage}{0.3\textwidth}
\centering
\begin{tikzpicture}[scale=0.3]
    \node[] (P4) at (4.6, 2) {\small $p_4$};
     \node[] (P3) at (4.6, -2) {\small $p_3$};
    \node[] (P1) at (-4.6, 2) {\small $p_1$};
     \node[] (P2) at (-4.6, -2) {\small $p_2$};
    \node[fill=black, circle, inner sep=1.5pt]
    (A) at (2.5, 0) {};
     \node[fill=black, circle, inner sep=1.5pt] (B) at (-2.5, 0) {};
     \node[fill=gray, circle, inner sep=1.5pt] (A1) at (5.5, 1.8) {};
      \node[fill=gray, circle, inner sep=1.5pt]  (A2) at (5.5, -1.8) {};
    \node[fill=gray, circle, inner sep=1.5pt] (B1) at (-5.5, 1.8) {};
      \node[fill=gray, circle, inner sep=1.5pt] (B2) at (-5.5, -1.8) {};
     \draw[bend left=70] (A) to (B) node[midway, above=-45 pt] {\Huge $m_2$};
    \draw (A) -- (A1);
    \draw (A) -- (A2);
       \draw (B) -- (B1);
    \draw (B) -- (B2);
\end{tikzpicture}
\end{minipage}
\hspace{12pt}
\begin{minipage}{0.3\textwidth}
\centering
\begin{tikzpicture}[scale=0.3]
    \node[] (P4) at (4.6, 2) {\small $p_4$};
     \node[] (P3) at (4.6, -2) {\small $p_3$};
    \node[] (P1) at (-4.6, 2) {\small $p_1$};
     \node[] (P2) at (-4.6, -2) {\small $p_2$};
    \node[fill=black, circle, inner sep=1.5pt]
    (A) at (2.5, 0) {};
     \node[fill=black, circle, inner sep=1.5pt] (B) at (-2.5, 0) {};
     \node[fill=gray, circle, inner sep=1.5pt] (A1) at (5.5, 1.8) {};
      \node[fill=gray, circle, inner sep=1.5pt]  (A2) at (5.5, -1.8) {};
    \node[fill=gray, circle, inner sep=1.5pt] (B1) at (-5.5, 1.8) {};
      \node[fill=gray, circle, inner sep=1.5pt] (B2) at (-5.5, -1.8) {};
     \draw[bend right=70] (A) to (B) node[midway, above=45 pt] {\Huge $m_1$};
    \draw (A) -- (A1);
    \draw (A) -- (A2);
       \draw (B) -- (B1);
    \draw (B) -- (B2);
\end{tikzpicture}
\end{minipage}
\hspace{12pt}
\begin{minipage}{0.3\textwidth}
\centering
\begin{tikzpicture}[scale=0.3]
    \node[] (P4) at (4.6, 2) {\small $p_4$};
     \node[] (P3) at (4.6, -2) {\small $p_3$};
    \node[] (P1) at (-4.6, 2) {\small $p_1$};
     \node[] (P2) at (-4.6, -2) {\small $p_2$};
    \node[fill=black, circle, inner sep=1.5pt]
    (A) at (2.5, 0) {};
     \node[fill=black, circle, inner sep=1.5pt] (B) at (-2.5, 0) {};
     \node[fill=gray, circle, inner sep=1.5pt] (A1) at (5.5, 1.8) {};
      \node[fill=gray, circle, inner sep=1.5pt]  (A2) at (5.5, -1.8) {};
    \node[fill=gray, circle, inner sep=1.5pt] (B1) at (-5.5, 1.8) {};
      \node[fill=gray, circle, inner sep=1.5pt] (B2) at (-5.5, -1.8) {};
    \draw (A) -- (A1);
    \draw (A) -- (A2);
       \draw (B) -- (B1);
    \draw (B) -- (B2);
\end{tikzpicture}
\end{minipage}
     \end{center}
     The spanning trees yield the polynomial
        $\mathcal{U}(x) = x_1 + x_2$ which has degree $\ell=1$.
        The second Symanzik polynomial has degree $\ell+1=2$, and it
       depends on three kinematic parameters:
        \begin{equation}
    \label{Eq:FRunning}
        \mathcal{F}_z(x) \,\,=\,\,    \mathcal{U}(x) \cdot
         (m_1 x_1 + m_2 x_2)- s \cdot x_1 x_2 
        \,\,=\,\,  m_1x_1^2 \,+\,  (m_1 \!+\! m_2 \!- \!s)\cdot x_1 x_2 
        \,+\, m_2 x_2^2.
    \end{equation}
    A parameter vector
        $z=(m_1,m_2,s) \in \RR^3$ satisfies
         $\mathcal{F}_z(x) \geq 0$ for all $x \in \RR^2_{\geq 0}$
      if and only if 
   \begin{equation}
   \label{eq:disjunction}
\quad   m_1,m_2 \geq 0 \,\,\quad {\rm and} \quad\,\,
 \bigl[ \,\, m_1+m_2 \geq s \,\,\,\,{\rm or} \,\,\,\,  4m_1m_2 \geq (m_1+m_2-s)^2 \,\bigr].
 \end{equation}
    This condition describes a closed convex cone in $\RR^3$.
    This is
    the  copositivity cone   $\mathcal{C}_G$ 
    for    (\ref{Eq:Bubble}).
    
The integral in (\ref{eq:univariateintegral2})
equals the Feynman integral $I_G(z)$ for $m_1=m_2 = 1$
and $s = 4- \epsilon $.
The point $z$ is in the interior of $\mathcal{C}_G$
if and only if $\epsilon > 0$.
This can be made manifest
by {\em P\'olya's method} \cite{Polya},
namely we certify copositivity by
showing that $\mathcal{F}_z(x)$ times 
 $(x_1 + x_2)^N$ has positive coefficients,
 for $N \gg 0$.
  By \cite[page 222]{PowersReznick},
 the smallest integer we can take is
 $$ N \, = \, 2 \left\lceil \frac{2}{\epsilon} \right\rceil -3. $$

See Figure \ref{FIG:BubbleC} for a cross section
of the copositive cone $\mathcal{C}_G$.
The left diagram shows the set $\mathcal{C}_G \cap \{3m_1+3m_2 = 6+s\}$.
This overlapping union of a triangle and an ellipse 
is not a basic semi-algebraic set.
Thus $\mathcal{C}_G$ is 
a spectrahedral shadow but it is
not a spectrahedron.
\end{example}

\begin{figure}[t]
\centering
\includegraphics[scale=0.42]{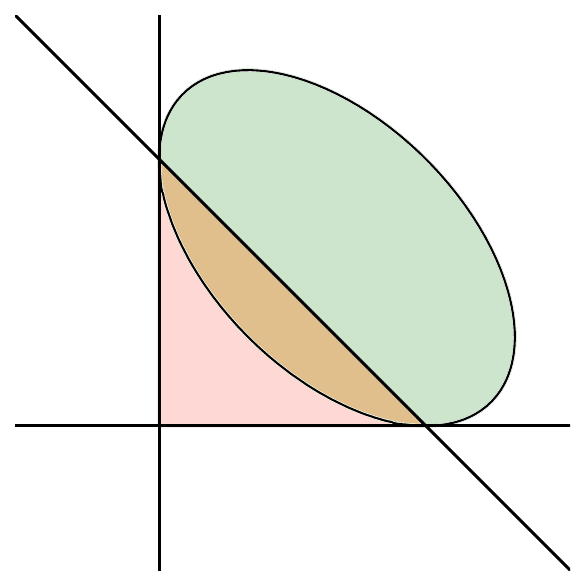} \qquad \qquad \qquad
\includegraphics[scale=0.61]{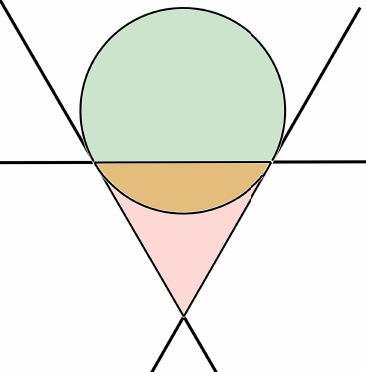} \vspace{-0.15cm}
\caption{The copositive cone $\mathcal{C}_G$ for the bubble diagram. The left diagram shows the intersection of $\mathcal{C}_G$ with the plane $\{3m_1 + 3m_2 = 6 + s\}$. The triangle 
is the region $\{m_1 \geq 0, \, m_2 \geq 0, \, m_1  + m_2 \geq s\}$, and the ellipse is   $\{m_1 \geq 0, \, m_2 \geq 0, \, 4m_1m_2 \geq (m_1 + m_2 -s)^2\}$. The right diagram is the same after the coordinate change $c_1 = m_1, \, c_2 = m_2, \, c_3 = m_1+m_2-s$.}
\label{FIG:BubbleC}
\end{figure}

We now present the organization of this paper, and we highlight our main results.
Section~\ref{sec2} gives a self-contained introduction to Feynman graphs, 
their kinematic parameters $z$, and the two Symanzik polynomials.
These specify the integrand in the Feynman integral $I_G(z)$.
Generalizing  Example \ref{ex:bubble}, we discuss how the
convergence of $I_G(z)$  depends on~$z$.

Most of the literature on copositive optimization
centers around quadratic forms and 
symmetric matrices. Even in this case, 
copositive geometry is highly nontrivial;
see e.g.~\cite{BomzeDur, deKlerkPasechnik, Vargas}.
For instance, testing membership in $\mathcal{C}_G$ is
an NP-hard problem.  In Section \ref{sec3} we transfer
this hardness to the physics context. In
Theorem \ref{thm:universality}, we prove that every quadratic form arises from 
a Feynman graph with one loop for some choice of
kinematic parameters.

Section \ref{sec4} offers case studies for
small Feynman graphs which are popular
in the literature, notably at the interface of physics
and number theory. The simplest of these
are the banana graphs,  which relate to
 Calabi-Yau varieties 
\cite[Chapter 14]{Weinzierl}.
Theorem~\ref{thm:banana} gives an explicit description of
the copositive cone for any banana graph.
We also explore the double box, non-planar double box, and the beetle, as in
\cite[Section IV]{CKN}. For us, all particles are massive.

In Section \ref{sec5} we develop copositive geometry
in the general setting of Gel'fand, Kapranov and Zelevinsky \cite{GKZ}.
For any support set $\mathcal{A}$, the copositive cone 
$\mathcal{C}_\mathcal{A}$ is
dual to the cone~over the positive toric variety, and its boundary is
contained in the principal  $\mathcal{A}$-determinant.
The specialization to Feynman polynomials leads
us to the theory in \cite{FMT, MizeraTelen}. We show that
the algebraic boundary of $\mathcal{C}_G$
is a subvariety of the principal Landau determinant.
The distinction between $\mathcal{C}_\mathcal{A}$ and 
$\mathcal{C}_G$ is worked out in detail for the parachute graph
(Example~\ref{Ex:Parachute}).

In Section \ref{sec6} we prove that P\'olya certificates \cite{Polya} exist
for all interior points of the copositive cone $\mathcal{C}_G$.
This rests on the special structure of Feynman polytopes, which
makes our problem amenable to  {\em Pólya's Theorem with Zeros}, due to
Castle, Powers and Reznick \cite{CastlePowersReznick}.
Theorem~\ref{Thm:PolyaSymanzik} states these apply to
all Feynman graphs $G$ and all kinematic parameters $z$.
The dedication of this paper recognizes the
importance of Vicki Powers' contributions \cite{CastlePowersReznick, PowersReznick}.

For any Feynman graph $G$ and specific kinematic parameters $z$, one seeks to
decide whether $z$ lies in $\mathcal{C}_G$ or not.
In either case, the outcome should be made manifest.
Section~\ref{sec7} is devoted to an algorithm for
making that decision, and for producing the desired certificates.
We focus on practical tools, and we present a proof-of-concept implementation of our algorithm in the Julia package \texttt{CopositiveFeynman.jl}. The code is made available
	 in the {\tt MathRepo} collection at MPI-MiS via
	\url{https://mathrepo.mis.mpg.de/CopositiveFeynman}.

This article points to many possibilities for future research.
These promise new connections  between
theoretical physics, polynomial optimization, and applied algebraic geometry.

\section{Graphs, Polynomials, and Integrals}
\label{sec2}

In this section, we review some basics on Feynman integrals,
following the text book \cite{Weinzierl}
and the articles \cite{Borinsky, BorinskyMunchTellander, HennRaman, MizeraTelen}.
There are many equivalent ways to write a Feynman integral \cite[Section 2.5]{Weinzierl}. 
We use the {\em Feynman parameter representation}, in which the integral reads
\begin{align}
\label{Eq:FeynmParamInt}
    I_G(z) \,\,= \,\,\frac{\Gamma\!\left(\,\sum_{i=1}^n \nu_i- \ell D / 2\right)}{\prod_{i=1}^n \Gamma\left(\nu_i\right)} \int_{\mathbb{P}_{>0}^{n-1}} \frac{\left(\,\prod_{i=1}^n x_i^{\nu_i} \right)\mathcal{U}(x)^{|\nu|-(\ell+1) D /2}}{\mathcal{F}_z(x)^{|\nu| - \ell D / 2}} \,\Omega.
\end{align}
An example with $n=2,\ell = 1,\nu = (2,2)$ and $D = 2$ can be seen in
equations (\ref{eq:univariateintegral2}) and (\ref{Eq:Bubble}).

We now explain the parts of \eqref{Eq:FeynmParamInt}.
The integration domain is the \emph{positive projective orthant}
\begin{align}
    \mathbb{P}_{>0}^{n-1} \,\,:=\,\, 
  \left\{\, [x_1 \colon \cdots \,\colon\, x_n] \in \mathbb{P}_{\mathbb{R}}^{n-1} \mid x_1 > 0, \dots , x_n > 0 \, \right\}, 
\end{align}
and we integrate against the differential form $\Omega=\sum_{i=1}^n(-1)^{n-i} \frac{\mathrm{~d} x_1}{x_1} \wedge \cdots \wedge \frac{\widehat{\mathrm{~d} x_i}}{x_i} \wedge \cdots \wedge \frac{\mathrm{~d} x_n}{x_n}$.
The notation $\tfrac{\widehat{\mathrm{d} x_i}}{x_i}$ indicates that $\tfrac{\mathrm{d} x_i}{x_i}$ is excluded from the product of $1$-forms.
The exponents $\nu_1, \dots, \nu_n$ are nonnegative real numbers,  and $|\nu| := \nu_1 + \dots + \nu_n$. The integer $D$ is the \emph{space-time dimension}. 
The prefactor is a constant, written with
Euler’s Gamma function~$\Gamma$. 

The integrand in \eqref{Eq:FeynmParamInt} is derived from a \emph{Feynman graph}. 
This is a graph  $G$ with vertex set $V$ and edge set $E$.  The \emph{valency} of a vertex $v \in V$ is the number of edges $e \in E$ incident to $v$.
If $v$ has valency $1$, then $e_v$ denotes the unique edge adjacent to $v$.
We call $e_v$ an \emph{external edge}. All other edges $e \in E$ are \emph{internal edges}.
Given a graph $G$ with $N$ external edges and $n$ internal edges, we associate \emph{momentum vectors} $p_1, \dots, p_N \in \mathbb{R}^D$ to the external edges and \emph{squared internal masses} $m_1, \dots, m_n \in \mathbb{R}$ to the internal edges.
For simplicity, in this article, we refer to the parameters $m_1, \dots, m_n$ as \emph{internal masses}, slightly deviating from the terminology commonly used in standard physics literature.
A \emph{Feynman graph} is a connected graph $G$ equipped with external momentum vectors and internal masses. When these are clear from the context, we may 
identify the Feynman graph with the graph $G$. 

We write $\ell$ for the number of independent cycles in the graph $G$. In the context of Feynman graphs, $\ell$ is called the number of \emph{loops} of $G$. Since $G$ is connected, we have
\begin{align*}
    \ell\,\, =\,\, \# E - \#V + 1.
\end{align*}
It remains to define the \emph{graph polynomials} $\mathcal{U}(x)$ and $ \mathcal{F}_z(x)$. 
The variables $x_1,\ldots,x_n$ are associated to the internal edges of $G$.
A \emph{spanning tree} of $G$ is a connected subgraph without cycles that contains all vertices of $G$.
  We write $\mathcal{T}$ for the set of all spanning trees of $G$. 
The \emph{first Symanzik polynomial} is the following sum of squarefree monomials,
all having degree $\ell$:
\begin{align*}
    \mathcal{U}(x) \,\,:=\,\, \sum_{T \in \mathcal{T}} \prod_{e \notin T} x_e.
\end{align*}
In matroid theory, we view $\mathcal{U}(x)$ as the sum over all bases of the cographic matroid of $G$.

A \emph{spanning $2$-forest} is a subgraph of $G$ that contains all vertices and all external edges, does not contain a cycle, and has two connected component. We write $T_1,T_2$ for the two connected components of a spanning $2$-forest. The set of all spanning $2$-forests is denoted by $\mathcal{W}$.
For a spanning $2$-forest $\{T_1,T_2\} \in \mathcal{W}$,
let $I_{T_1}, I_{T_2} \subset [N]$ be the index sets of the external edges that are attached to $T_1$ and $T_2$ respectively. The \emph{second Symanzik polynomial}~is 
\begin{align}
\label{Eq:F}
    \mathcal{F}_z(x)\,\, := \sum_{\{T_1,T_2\} \in \mathcal{W}}\!\biggl( \,\sum_{i \in I_{T_1}} \sum_{j \in I_{T_2}} k_{ij} 
    \biggr) \!\! \prod_{e \notin T_1 \sqcup T_2} \!\! \!x_e \,
    \,+ \,\,\biggl(\,\sum_{e=1}^n m_e x_e \biggr)\, \mathcal{U}(x).
\end{align}
Here, $k_{ij} = p_i \cdot p_j = p_{i1}p_{j1} -  p_{i2}p_{j2}-\dots -  p_{iD}p_{jD}$ 
denotes the \emph{Minkowski scalar product}. Note that $
\mathcal{F}_z(x)$ is homogeneous  of degree $\ell + 1$, because
  each spanning $2$-forest is obtained by removing $\ell + 1$ internal edges from $G$.
  The monomials that appear give two nice polytopes.

\begin{lemma} \label{lem:NimaSebastian}
The Newton polytopes of the two Symanzik polynomials have
the following descriptions:
 ${\rm Newt}(\mathcal{U})$ is the
matroid polytope of the cographic matroid of $G$, while
$\Newt(\mathcal{F})$ is 
the Minkowski sum of the matroid polytope $\Newt(\mathcal{U})$ 
with the simplex $\Delta_{n-1} = \Newt(\sum_{i=1}^n x_i)$.
\end{lemma}

\begin{proof}
This follows from the definitions of these polytopes.
See also \cite[Section IV.A]{AHM}.
\end{proof}

  We regard $\mathcal{F}_z(x)$ as a parametrized polynomial.
  The variables are $  x_1, \dots, x_n$ and the parameter
  vector  $z$  has entries $k_{ij}$ and $m_e$. 
The Feynman integral $I_G(z)$ is a function of
these parameters. The integral often diverges.
Whether this happens or not depends on~$z$.
The following sufficient condition for convergence 
was given by Borinsky in \cite[Theorem 3]{Borinsky}.

\begin{theorem} \label{Prop:ConvergInER}
Fix a Feynman graph $G$, and choose $D \in \NN$ and
$\nu \in \mathbb{R}^n_{>0}$ such that
\begin{equation}
\label{eq:polytopecontainment}
 \nu \,+\, (|\nu|-(\ell\!+\!1) D /2) \Newt \! \big(\mathcal{U}(x)\big) 
\,\, \subseteq \,\,\relint\Big( (|\nu|-\ell D /2) \Newt \! \big(\mathcal{F}_z\big)\Big).
\end{equation}
If $z $ is in the interior of
the  copositive cone $\mathcal{C}_G$
 then the Feynman integral in \eqref{Eq:FeynmParamInt}  converges.
\end{theorem}
For $\nu = (2,2)$ and $ D=2$, up to a prefactor,  $I_G(z)$ is the integral~\eqref{eq:univariateintegral2}, which we discussed in the introduction.
The two polytopes in \eqref{eq:polytopecontainment} are $\conv ((2,4),(4,2))$ and $\conv((6,0),(0,6))$. 

The convergence in Theorem~\ref{Prop:ConvergInER} rests on two hypothesis.
First, there is combinatorics of polytopes, in the
containment  relation (\ref{eq:polytopecontainment}).
We can achieve this by choosing $D$ and $\nu$.
Second, there is copositive geometry: the denominator polynomial $\mathcal{F}_z$
must be strictly copositive. In \cite{Borinsky} and other sources, this is
phrased as saying that $\mathcal{F}_z$ is {\em completely non-vanishing} on $\PP^{n-1}_{>0}$.
We will see
in Section \ref{sec5} that this happens if $z$ lies in the interior of the  cone $\mathcal{C}_G$.

Theorem \ref{Prop:ConvergInER} justifies our argument that
the copositive cone replaces what is called
the {\em Euclidean region} in the physics literature;
see e.g.~\cite[Section 2.2]{BorinskyMunchTellander}
and \cite[Section 2.5.1]{Weinzierl}.
Here the authors require that $\mathcal{F}_z$ has
positive coefficients. If this holds then $z$
is in~$\mathcal{C}_G$, but for trivial reasons.
More interesting are points
$z \in {\rm int}(\mathcal{C}_G)$ for which
$\mathcal{F}_z$ has some negative coefficients.
This indicates scenarios where the Feynman integral  converges unexpectedly.

We now take a closer look at the parameter space
in which the copositive cone~$\mathcal{C}_G$ lives.
The internal masses $m_1,\ldots,m_n$ are real and otherwise
unconstrained.   The Gram matrix $K = (k_{ij})$ is symmetric and of size $N \times N$.
We assume that its row sums are zero, so the matrix has
$\binom{N}{2}$ indepenent entries.
 This reflects the {\em momentum conservation} assumption $p_1 + \dots + p_N = 0$. 
In physics, the Gram matrix is also Lorentzian and of rank at most~$D$.
The semialgebraic constraints this would impose are studied in \cite{CFS}.
We here relax all inequalities and all rank constraints, and we
view $z = (m_e,k_{ij})$ as an arbitrary real vector of
length $\binom{N}{2}+n$. In other words, we identify our parameter space
 with the real vector space $\RR^{\binom{N}{2}+n }$.
 
 \begin{proposition}
$\mathcal{C}_G$ is a
full-dimensional closed semi-algebraic  convex cone in~$\RR^{\binom{N}{2} + n}$.
   \end{proposition}

\begin{proof} For any $u \in \RR^n_{\geq 0}$, the inequality
 $\mathcal{F}_z(u) \geq 0$ is linear in $z$. By definition, the set
   $\mathcal{C}_G$ consists of  all solutions to this infinite system of linear
 inequalities. It is therefore a closed convex cone.
Tarski's Theorem on Quantifier Elimination implies
that  $\mathcal{C}_G$ is semi-algebraic, i.e.~it can
be described by a finite Boolean combination of
polynomial inequalities in $z$. Furthermore, the  cone
$\mathcal{C}_G$ is full-dimensional because it contains 
an orthant $\RR^{\binom{N}{2}+n}_{> 0}$.  Equivalently,
the polynomial in~(\ref{Eq:F}) is copositive
when all of its coefficients are positive. 
\end{proof}

\begin{remark}
The copositive cone $\mathcal{C}_G$ is generally not
a pointed cone. It contains a linear subspace of positive dimension.
For instance, for the bubble diagram in Example~\ref{ex:bubble}, 
the parameter space is $ \RR^{8}$, with two coordinates
$m_i$ and six coordinates $k_{ij}$.
The $8$-dimensional cone $\mathcal{C}_G$ contains the $5$-dimensional
linear space defined by
$m_1 = m_2 = s = 0$. Here we set
$s = -k_{13} - k_{14} - k_{23} - k_{24}$,
as seen from (\ref{Eq:FRunning}) and  (\ref{Eq:F}).
In practice, we work modulo the lineality space.
Thus, for the bubble diagram, we regard
$\mathcal{C}_G$ as the $3$-dimensional cone over Figure \ref{FIG:BubbleC}.
\end{remark}

\section{Copositive Matrices and One-Loop Diagrams}
\label{sec3}

The adjective {\em copositive} was coined by
Motzkin in his 1952 paper \cite{Motzkin}.
We write $\mathcal{C}_{n,d}$ for the \emph{copositive cone}, which consists of the copositive polynomials
in the space $\RR[x_1,\ldots,x_n]_d \simeq 
\RR^{\binom{n+d-1}{d}}$ of all homogeneous
polynomials of degree $d$ in $n$ variables.
The study of $\mathcal{C}_{n,d}$ is an active area of research
in optimization; see \cite[Section 9]{Nie:Book}.
Most of that literature centers around the case of quadratic forms $(d=2)$.
A symmetric $n \times n$ matrix $C$ is called {\em copositive} if its
associated quadratic form $f(x) = x^{\top} C x$ is copositive \cite{Hadeler}.
Thus $\mathcal{C}_{n,2}$ is the copositive cone in 
the space $\RR^{\binom{n+1}{2}}$ of symmetric matrices.

The cone $\mathcal{C}_{n,2}$ is a surprisingly complicated object.
Deciding whether a given symmetric matrix $C$ lies in $\mathcal{C}_{n,2}$
is an NP-hard problem. See \cite[Section 4]{Dur}. A wide range of
combinatorial optimization problems can be reduced to this membership problem.
Bomze et~al. \cite{BomzeDur} introduced the term 
\emph{copositive programming} for  
optimization problems modeled on~$\mathcal{C}_{n,2}$.

The copositive cone $\mathcal{C}_{n,2}$ has two natural subcones,
namely the  orthant  $\RR^{\binom{n+1}{2}}_{\geq 0}$ of matrices
with nonnegative entries and the cone
${\rm PSD}_n$ of positive semidefinite matrices. We have
\begin{equation}
\label{eq:coneinclusion}
\RR^{\binom{n+1}{2}}_{\geq 0}  \,+ \, 
{\rm PSD}_n \,\, \subseteq \,\, 
\mathcal{C}_{n,2}.
\end{equation}
A matrix is manifestly copositive if it is
the sum of a nonnegative matrix and
a positive semidefinite matrix. This certificate
always works for $n \leq 4$. Indeed, it is
known that the equality holds in (\ref{eq:coneinclusion})
for $n=2,3,4$.
However, this fails for $n \geq 5$. Here is a famous example.

\begin{example}[Horn is manifestly copositive] \label{ex:horn}
Let $n=5$ and consider the quadratic form
 \[ h \,=\, x_1^2 + x_2^2 + x_3^2 + x_4^2 + x_5^2
       -2 (x_1x_2 + x_2x_3 + x_3x_4 + x_4x_5 + x_5x_1) +2 (x_1x_3 + x_2x_4 + x_3x_5 + x_4x_1 + x_5x_2).\]
The corresponding symmetric  $5 \times 5$ matrix is known as the {\em Horn matrix}. It  equals
$$
    H \,\,=\,\, \begin{footnotesize} \begin{pmatrix}
        \phantom{-}1 & -1 &  \phantom{-}1 &  \phantom{-}1 & -1 \,\,\\
        -1 &  \phantom{-}1 & -1 &  \phantom{-}1 &  \phantom{-}1 \,\,\\
         \phantom{-}1 & -1 &  \phantom{-}1 & -1 &  \phantom{-}1 \,\,\\
         \phantom{-}1 &  \phantom{-}1 & -1 &  \phantom{-}1 & -1 \,\, \\
        -1 &  \phantom{-}1 &  \phantom{-}1 & -1 &  \phantom{-}1 \,\,
    \end{pmatrix}. \end{footnotesize}
$$
It is known that $H$ is not
in the left hand side of
(\ref{eq:coneinclusion}).
However, $H$ is  copositive, because
$$
\begin{matrix} 4(
 x_1 x_2 x_4+
 x_2 x_3 x_5+
 x_3 x_4 x_1+
 x_4 x_5 x_2+
 x_5 x_1 x_3) \,+\,
 x_1 (x_1-x_2+x_3+x_4-x_5)^2 \\ \,\,+ \,\,
 x_2 (x_2-x_3+x_4+x_5-x_1)^2 \,+\,
 x_3 (x_3-x_4+x_5+x_1-x_2)^2 \\\, \,\,+\,\,
 x_4 (x_4-x_5+x_1+x_2-x_3)^2 \,+\,
 x_5 (x_5-x_1+x_2+x_3-x_4)^2 
  \end{matrix}
$$ 
factors into $\, h(x) \cdot (x_1+x_2+x_3+x_4+x_5)  $.
This formula is due to Parrilo \cite[Section 5.4]{Parrilo}.
\end{example}

 Example \ref{ex:horn} suggests that copositivity 
of a quadratic form $h$ can be made manifest by writing
$\, h(x_1^2,\ldots,x_n^2) \cdot (x_1^2 + \cdots + x_n^2)^r \,$
as a sum of squares (SOS) for some $r \in \NN$.
De Klerk and Pasechnik \cite{deKlerkPasechnik} 
used this to compute stability numbers of graphs.
Schweighofer and Vargas \cite{SchweighoferVargas} proved that an SOS representation always
exists for $n=5$, but it was shown 
in \cite{LaurentVargas} that this no longer works for $n \geq 6$.
A recent advance by Bodirsky, Kummer and Thom \cite{BKT} 
reveals that $\mathcal{C}_{n,2}$ is not a spectrahedral shadow for $n \geq 5$.
In spite of these negative results, we still seek practical tools for certifying copositivity.
In Section~\ref{sec6} we examine P\'olya's classical method~\cite{CastlePowersReznick, Polya}.

For the remainder of this section, we return to our physics application.
In Theorem \ref{thm:universality} we will show that every quadratic form can be
realized by a Feynman graph with $\ell=1$. This can be viewed as a
hardness result for testing the convergence of Feynman integrals.

We now  focus on \emph{one-loop Feynman graphs} with $n$ internal edges and $n$ external edges:
\begin{equation}
\label{Eq:Bubble2}
\begin{aligned}
\begin{tikzpicture}[scale=0.5]
          \node[fill=black, circle, inner sep=1.5pt]
    (D1) at (-1.5, 1) {};
       \node[fill=gray, circle, inner sep=1.5pt]  (A1) at (-3, 2.5) {};
         \draw (D1) -- (A1);
             \node[] (P1) at (-3.7, 2.5) {$p_1$};;
              \node[fill=black, circle, inner sep=1.5pt]
    (D2) at (0, 2) {};
           \draw (D1) to (D2);
            \node[] (P1) at (-1, 1.9) {$m_2$};
     \node[fill=black, circle, inner sep=1.5pt]
    (D3) at (1.5, 1) {};
       \node[fill=black, circle, inner sep=1.5pt]
    (D4) at (1.5, -1) {};
      \node[fill=black, circle, inner sep=1.5pt]
    (D5) at (0, -2) {};
       \node[fill=black, circle, inner sep=1.5pt]
    (D6) at (-1.5, -1) {};
      \node[fill=gray, circle, inner sep=1.5pt]  (A2) at (0, 3.5) {};
         \draw (D2) -- (A2);
             \node[] (P2) at (0.7, 3.5) {$p_2$};
               \node[] (M3) at (1, 1.9) {$m_3$};
              \node[fill=black, circle, inner sep=1.5pt]
    (D2) at (0, 2) {};
           \draw (D1) to (D2);
           \draw (D2) to (D3);
           \draw (D3) to (D4);
           \draw[dashed] (D4) to (D5);
           \draw (D5) to (D6);
           \draw (D1) to (D6);
            \node[fill=black, circle, inner sep=1.5pt]
    (D1) at (1.5, -1) {};
       \node[fill=gray, circle, inner sep=1.5pt]  (A4) at (3, -2.5) {};
         \draw (D4) -- (A4);
             \node[] (P4) at (3.7, -2.5) {$p_4$};;
               \node[fill=gray, circle, inner sep=1.5pt]  (A5) at (0, -3.5) {};
         \draw (D5) -- (A5);
             \node[] (P5) at (-0.1, -4) {$p_{n-1}$};
               \node[] (M5) at (-1, -1.9) {$m_n$};
                    \node[fill=gray, circle, inner sep=1.5pt]  (A3) at (3, 2.5) {};
         \draw (D3) -- (A3);
             \node[] (P6) at (-3.7, -2.5) {$p_n$};
              \node[fill=gray, circle, inner sep=1.5pt]  (A6) at (-3, -2.5) {};
         \draw (D6) -- (A6);
             \node[] (P3) at (3.7, 2.5) {$p_3$};
                \node[] (M1) at (-2.2, 0) {$m_1$};
                 \node[] (M4) at (2.2, 0) {$m_4$};
\end{tikzpicture}
\end{aligned}
\end{equation}
The second Symanzik polynomial for the graph in (\ref{Eq:Bubble2})
is the quadratic form
 \begin{align}
 \label{Eq:OneloopF}
  \mathcal{F}_z(x) \,\,\,= \,\,\,
  \sum_{i=1}^n m_i x_i^2 \,+\, \sum_{i=1}^n \sum_{j=i+1}^n \Big( m_i + m_j + \sum_{a \in I_{i}^j} \sum_{b \in [n] \setminus I_{i}^j} k_{ab} \Big)x_i x_j,
    \end{align}
    where $I_i^j = \{i,\dots,j-1\}$ for $i < j$, and $[n] = \{1, \dots ,n\}$. Our result states that every quadratic polynomial can be obtained as a second Symanzik polynomial for some choice of the kinematic parameters.

\begin{theorem} \label{thm:universality}
 For every quadratic form $f \in \mathbb{R}[x_1, \dots, x_n]_2$, 
 the graph in (\ref{Eq:Bubble2}) admits
   kinematic parameters $z = (k_{ij},m_e)$ 
   such that $f(x)$ equals
   the second Symanzik polynomial    $\mathcal{F}_z(x)$.
\end{theorem}

\begin{proof}
We start with an arbitrary symmetric $n \times n$ matrix $C = (c_{ij})$ and we 
write
    \begin{align*}
    f(x) \,\,=\,\, \sum_{i=1}^n \sum_{j=i}^n c_{ij}x_i x_j
    \end{align*}
We set  $m_i = c_{ii}$ for the masses. With this,
  the coefficients of
          $x_i^2$ in $f(x)$ and in $\mathcal{F}_z(x)$ agree.
To match the remaining coefficients, we must solve the following system of linear equations:
\begin{equation}
\label{eq:tomatch}
                \sum_{a \in I_{i}^j} \sum_{b \in [n] \setminus I_{i}^j} \! k_{ab}
        \,\, = \,\, c_{ij} - c_{ii} - c_{jj}
        \qquad \text{for } \,
        1\leq i < j \leq n.
\end{equation}
Here $K = (k_{ij})$ is an unknown symmetric $n \times n$ matrix
with zero row sums and zero column sums, since we assume
momentum conservation. 
The space of such matrices $K$ is the domain for the
linear map $\RR^{\binom{n}{2}} \rightarrow
\RR^{\binom{n}{2}}$  given 
by the left hand side of (\ref{eq:tomatch}).
Our claim states that this linear map is surjective.
It suffices to show that
every $k_{ab}$ is a linear combination of 
    \begin{align*}
        M_{i,j}\,\, :=\,\,  \sum_{u \in I_{i}^j} \sum_{v \in [n] \setminus I_{i}^j} \!  k_{uv}
        \qquad {\rm for} \quad 1 \leq i < j \leq n.
    \end{align*}
    We prove this claim by induction on $a-b$. For $a-b = 0$, momentum conservation implies
    \begin{align*}
       - M_{a,a+1} \,\,\, = \,\,\,-\! \!\!\sum_{ d \in [n] \setminus \{a\} } k_{ad} \,\,= \,\,k_{aa}.
    \end{align*}
    
  Suppose each $k_{ab}$ with $|a-b| \leq q-1 $ is a linear combination of the $M_{i,j}$.
     In the induction step, we show that this is also true for $|a-b| = q  $. By momentum conservation, we have
    \begin{align*}
        M_{a,a+q +1} \,\,= \sum_{ c \in I_a^{a+q+1}} \sum_{ d \in [n] \setminus I_a^{a+q+1} }\! k_{cd} 
        \,\,\,=\,\,\, -\!\!\! \sum_{ c \in I_a^{a+q+1}} \sum_{ d \in I_a^{a+q+1} } \!k_{cd}.
    \end{align*}
This implies
      \begin{align*}
        k_{a,a+q} \,\,=\,\, \tfrac{1}{2} \Big( 
     -M_{a,a+q+1} \,\,-\!\! \sum_{ d \in I_a^{a+q} } \! k_{ad}   \,\,- \!\sum_{ c \in I_{a+1}^{a+q}} \sum_{ d \in I_a^{a+q+1} } \! \! k_{cd} \,\,\, -\!\!  \sum_{ d \in I_{a+1}^{a+q+1} } \!\!k_{a+q,d} \,\Big).
    \end{align*}
    By induction,      the right-hand side of this equation is a linear combination of the $M_{i,j}$'s.
    \end{proof}

\begin{example}[Pentagon] Fix the  one-loop diagram with five internal edges. 
Then we have
$$ \begin{small} \begin{matrix}
      \!\!\! \!\!\! \!\!\!  \!\!\!  \mathcal{F}_z(x) \,\,= \,\, m_1 x_1^2 + m_2x_2^2 + m_3x_3^2 + m_4x_4^2 + m_5x_5^2 
 + (m_1+m_2\!+\! k_{12}\!+\! k_{13}\!+\! k_{14}\!+\! k_{15})x_1x_2  \qquad \\
 +\, (m_2+m_3\!+\! k_{12}\!+\! k_{23}\!+\! k_{24}\!+\! k_{25})x_2x_3  
+ (m_3+m_4\!+\! k_{13}\!+\! k_{23}\!+\! k_{34}\!+\! k_{35})x_3x_4  \\
 +\, (m_4+m_5\!+\! k_{14}\!+\! k_{24}\!+\! k_{34}\!+\! k_{45})x_4x_5  
 + (m_1+m_5\!+\! k_{15}\!+\! k_{25}\!+\! k_{35}\!+\! k_{45})x_5x_1  \\
 + (m_1{+}m_3\!+\! k_{13}\!+\! k_{14}\!+\! k_{15}\!+\! k_{23}\!+\! k_{24}\!+\! k_{25})x_1x_3  
 + (m_1{+}m_4\!+\! k_{14}\!+\! k_{15}\!+\! k_{24}\!+\! k_{25}\!+\! k_{34}\!+\! k_{35})x_1x_4  \\
 + (m_2{+}m_4\!+\! k_{12}\!+\! k_{24}\!+\! k_{25}\!+\! k_{13}\!+\! k_{34}\!+\! k_{35})x_2x_4  
 + (m_2{+}m_5\!+\! k_{12}\!+\! k_{25}\!+\! k_{13}\!+\! k_{35}\!+\! k_{14}\!+\! k_{45})x_2x_5  \\
 +\, (m_3+m_5\!+\! k_{13}\!+\! k_{23}\!+\! k_{35}\!+\! k_{14}\!+\! k_{24}\!+\! k_{45})x_3x_5 .
      \end{matrix} \end{small}
$$
Theorem \ref{thm:universality} says that this covers every quadratic form.
The  Horn polynomial $h(x)$ from Example \ref{ex:horn}
equals the second Symanzik polynomial $\mathcal{F}_z(x)$ for
$m_1 \!=\! m_2 \!=\! \cdots \!=\! m_5 = 1$ and
$$ K \,\,\,=\,\,\, \begin{footnotesize} \begin{pmatrix}
            k_{11} & k_{12} & k_{13} & k_{14} &k_{15} \\
            k_{12} & k_{22} & k_{23} & k_{24} &k_{25} \\
            k_{13} & k_{23} & k_{33} & k_{34} &k_{35} \\
            k_{14} & k_{24} & k_{34} & k_{44}& k_{45} \\
            k_{15} & k_{25} & k_{35} & k_{45} &k_{55} \\
        \end{pmatrix}  \end{footnotesize}
          \,\,\,= \,\,\, \begin{footnotesize}\begin{pmatrix}
            \phantom{-}4 & -4 & \phantom{-}2 & \phantom{-}2 & -4  \,\, \\
            -4 & \phantom{-}4 & -4 & \phantom{-}2 & \phantom{-}2 \,\, \\
            \phantom{-}2 & -4 & \phantom{-}4 & -4 & \phantom{-}2\,\, \\
            \phantom{-}2 & \phantom{-}2 & -4 &\phantom{-} 4& -4\,\, \\
            -4 & \phantom{-}2 & \phantom{-}2 & -4 & \phantom{-}4\,\, \\
        \end{pmatrix}\!. \end{footnotesize}
$$
\end{example}
   
\smallskip
   
After Motzkin  \cite{Motzkin},
copositivity of matrices became a popular topic.
Numerous articles from the 20th century offer semialgebraic characterizations.
See the references in \cite{Hadeler}.
Each characterization amounts to a non-trivial Boolean combination
of inequalities on matrix entries and principal minors, and this is
consistent with the 21st century complexity results.

\begin{example}[Triangle diagram]
\label{Example:Triangle}
The triangle $G$ is~(\ref{Eq:Bubble2}) for $n=3$.
Every ternary quadric is the second Symanzik polynomial
 for some choice of kinematic parameters. We thus write
 $$ \mathcal{F}_z(x) \,\, = \, \,
 x^{\top} C x \,\,=\,\,
 \sum_{i=1}^3 \sum_{j=1}^3 c_{ij} x_i x_j .
 $$
  Hadeler \cite[Theorem 4]{Hadeler} tells us that the 
  $3 \times 3$ matrix $C = (c_{ij}) $ is copositive if and only if
\begin{equation}
\label{eq:trianglecondition}
 \begin{matrix} c_{11} \geq 0 \,, \,\,c_{22} \geq 0 \,,\, \, c_{33} \geq 0 \,, \,\,
   c_{12} \geq - \sqrt{c_{11} c_{22}}\,,\,\,
      c_{13} \geq - \sqrt{c_{11} c_{33}}\,,\,\,
      c_{23} \geq - \sqrt{c_{22} c_{33}}\,, \smallskip \\
       \hbox{and \quad} \bigl[ \, \,{\rm det}(C) \geq 0\,
   \quad   \hbox{or} \quad
            c_{12} \sqrt{ c_{33}} +
      c_{13} \sqrt{ c_{22}}  + 
      c_{23} \sqrt{ c_{11}}   + \sqrt{c_{11} c_{22} c_{33}} \,\geq\, 0 \, \bigr]. 
      \end{matrix} 
\end{equation}      
Furthermore, the algebraic boundary of $\mathcal{C}_G$ is
the product of the seven principal minors of~$C$.      
The disjunction in (\ref{eq:trianglecondition})
is analogous to (\ref{eq:disjunction}), which concerns the case $n=2$.
We invite our readers to extend these formulas to $4 \times 4$ matrices,
where $G$ is the {\em box diagram} ($n=4$).
\end{example}

\section{Bananas, Boxes and Beetles}
\label{sec4}

 The {\em banana graph} with $n$ internal edges is the simplest
Feynman graph of arbitrarily high genus $\ell$.
Namely, it has only two vertices connected by $n$ internal edges,
and thus $\ell = n-1$. 
In spite of this simplicity,  Feynman integrals
of bananas  lead to
deep mathematical structures.
In \cite[Section 14.5.3]{Weinzierl},
one starts from mixed Hodge structures and derives {\em banana motives}.

The second Symanzik polynomial of the banana graph $G$ with $n$ internal edges
is the following homogeneous polynomial of degree $n$ in $n$ variables $x_i$
with parameters $z = (m_i,s)$:
$$ \mathcal{F}_z (x) \,= \,
x_1 x_2 \cdots x_n \cdot \bigl( f(x) \cdot g(x) \, - \, s \bigr) , $$
where the factors in the parenthesis are the following linear form and reciprocal linear form:
$$f(x) \,= \,\sum_{i=1}^n m_i x_i \qquad {\rm  and}  \qquad g(x) \,=\, \sum_{j=1}^n \frac{1}{x_j}. $$
The parameters $\,m_1,\ldots,m_n\,$ and $\,s\,$ represent the
masses and momenta of the particles in a scattering process.
The hypersurface defined by $\mathcal{F}_z(x)=0$  in the projective space $\PP^{n-1}$ is
a Calabi-Yau variety. For instance, for $n=3$ this hypersurface
 is an elliptic curve in $\PP^2$. This curve depends on four parameters
$s,m_1,m_2,m_3$, and its discriminant is the quartic in~(\ref{eq:bananadiscriminant}).
For $n=4$, the variety  $\{ x \in \PP^3: \mathcal{F}_z(x) = 0\}$ is a K3 surface.
The Feynman integrals for bananas are period
integrals  \cite[Section 10.3]{Weinzierl} on  Calabi-Yau varieties. 
A Calabi-Yau variety of dimension $0$ consists of two points, as seen for $n=2$ in Example \ref{ex:bubble}.
We here prove:

\begin{theorem} \label{thm:banana}
The copositive cone for the banana graph equals
$$ \mathcal{C}_G \,\, = \,\,
\bigl\{ \,(m,s) \in \RR^{n+1} \,:\, m_1,\ldots,m_n \geq 0 \,\,\,\,{\rm and} \,\,\,
 \sum_{i=1}^n m_i \, + \, 2 \cdot \sum_{i < j}  \sqrt{ m_i} \sqrt{m_j } \, \,\geq \,\, s \,\bigr\}.
$$
It follows that
the algebraic boundary of $\,\mathcal{C}_G$ is the union of
the $n$ coordinate hyperplanes $\{m_i=0\}$ with an irreducible
hypersurface of degree $2^{n-1}$.
\end{theorem}

To prove this result, we consider the polynomial optimization problem
\begin{equation}
\label{eq:bananaopt}
 {\rm Minimize} \quad f(x) \cdot g(x) \quad \hbox{subject to} 
\,\,\, x \in \RR^n_{> 0} . \end{equation}
Since the Laurent polynomial $f(x) g(x)$ is homogeneous of degree zero,
we can replace the orthant $\RR^n_{>0}$
with the open probability simplex ${\rm int}(\Delta)$ in the
optimization problem (\ref{eq:bananaopt}).
The optimal value $v^* = v^*(m_1,\ldots,m_n)$ is an algebraic function of
the masses $m_1,\ldots,m_n$.  Note that $v^*(m)$ is positive
whenever the coordinates of $m = (m_1,\ldots,m_n)$ are positive.
We have:

\begin{lemma}
\label{Lemma:OptimalValue}
The optimal value in (\ref{eq:bananaopt}) is given by the following expression in the masses:
$$ v^*(m) \quad = \quad \sum_{i=1}^n m_i \, + \, 2 \cdot \sum_{i < j}  \sqrt{ m_i m_j }$$
\end{lemma}

\begin{proof}
The objective function $f(x) g(x)$ is positive on the open simplex
${\rm int}(\Delta)$ and it tends to infinity on the boundary. Hence the
gradient of $f(x) g(x)$ vanishes at the optimal point $x^*$. 
By the product rule from calculus, the gradient is the following row vector of length $n$:
$$ \begin{matrix} \nabla_x (f(x) g(x)) & = &
 g(x) \cdot \nabla_x(f(x)) \,+\, f(x)\cdot \nabla_x(g(x)) \smallskip \\ & = & 
g(x) \cdot \begin{bmatrix} m_1 & m_2 & \cdots & m_n \end{bmatrix}
\,-\, f(x) \cdot \begin{bmatrix} \frac{1}{x_1^2} & \frac{1}{x_2^2} & \cdots & \frac{1}{x_n^2} \end{bmatrix}.
\end{matrix}
$$
At the optimal point $x^*$, we have $g(x^*)m_i = f(x^*)/(x_i^*)^2$.
Setting $\lambda = \sqrt{\frac{f(x^*)}{g(x^*)}}> 0$, we find
$$ x_i^* \, = \, \frac{\lambda}{\sqrt{m_i}}\qquad
\hbox{for} \,\, i=1,2,\ldots,n. $$
Substituting this optimal point into the objective function yields the optimal value:
$$ v^*(m) \,\, = \,\,
f(x^*)\cdot g(x^*) \,\, = \,\, \biggl(\,\sum_{i=1}^n \sqrt{m_i}\, \biggr)^2
$$
This completes the proof of the lemma.
\end{proof}

\begin{proof}[Proof of Theorem \ref{thm:banana}]
By definition, a point $(m,s)$ lies in the copositive cone $\mathcal{C}_G$ if  and only if
the Laurent polynomial
$f(x) \cdot g(x) - s $ is nonnegative on
the open orthant $\RR^{n}_{>0}$. This happens if and only 
if the optimal value $v^*(m)$ in (\ref{eq:bananaopt}) is larger than or equal to $s$,
so $\, \mathcal{C}_G \, = \,
\{ \,(m,s) \in \RR^{n}_{\geq 0} \times \RR \,: \, v^*(m) \geq s \,\}$.
The first assertion now follows from Lemma~\ref{Lemma:OptimalValue}.

For the second assertion we rationalize the equation
$v^*(m) = s $.
To do this, we multiply the expressions
obtained from $v^*(m)-s$ by taking all $2^n$ combinations of positive
and negative square roots $\pm \sqrt{m_i}$. This leads to a polynomial
of degree $2^{n}$, which is the square of a polynomial of degree $2^{n-1}$,
because a global sign flip leaves $v^*(m)-s$ unchanged.
Using elementary Galois theory, one can show that this polynomial
of degree $2^{n-1}$ is irreducible.
\end{proof}

The boundary polynomial of degree $2^{n-1}$ is the
Landau discriminant~\cite{MizeraTelen} for the banana graph.
For example, for the two-loop banana  $(n=3)$, the algebraic boundary of $\mathcal{C}_G$ equals
$$ \prod_{i_1=0}^1 \prod_{i_2=0}^1\, \biggl[ \,s+m_1+m_2+m_3 + 2(-1)^{i_1+i_2} \sqrt{m_1} \sqrt{m_2} 
+ 2(-1)^{i_1} \sqrt{m_1} \sqrt{m_3} 
+ 2(-1)^{i_2} \sqrt{m_2} \sqrt{m_3} \,\biggr]. $$
This Landau discriminant is the quartic polynomial
\begin{equation}
\label{eq:bananadiscriminant} \begin{small}
 \begin{matrix} 
s^4\,-\, 4 \,( m_1+ m_2+ m_3) \, s^3\,+\, (6 m_1^2+4 m_1 m_2+4 m_1 m_3+6 m_2^2+4 m_2
 m_3+6 m_3^2) \,s^2 \\- \,4 \,( m_1^3- m_1^2 m_2- m_1^2 m_3- m_1 m_2^2+10 m_1 m_2 m_3- m_1 m_3^2+
  m_2^3- m_2^2 m_3- m_2 m_3^2+m_3^3) \, s  \\ + \,
  (\,m_1^4-4 m_1^3 m_2-4 m_1^3 m_3+6 m_1^2 m_2^2+4 m_1^2 m_2 m_3+6 m_1^2 m_3^2-4 m_1 m_2^3  
+4 m_1 m_2^2 m_3 \\ 
+4 m_1 m_2 m_3^2-4 m_1 m_3^3+m_2^4-4 m_2^3 m_3+6 m_2^2 m_3^2-4 m_2 m_3^3+m_3^4\,).
\end{matrix} \end{small}
\end{equation}

Emboldened by Theorem \ref{thm:banana}, we next consider three
Feynman graphs with $\ell=2$, namely the 
{\em double box}, the {\em nonplanar double box} and the {\em beetle}.
The three graphs are as follows:
\begin{equation}
\label{Eq:BoxBeetle}
\begin{aligned}
\begin{minipage}{0.3\textwidth}
\centering
\begin{tikzpicture}[scale=0.6]
    \node[fill=black, circle, inner sep=1.5pt] (C) at (0, 1) {};
    \node[fill=black, circle, inner sep=1.5pt] (D) at (0, -1) {};
        \node[fill=black, circle, inner sep=1.5pt] (E) at (2, 1) {};
    \node[fill=black, circle, inner sep=1.5pt] (F) at (2, -1) {};
        \node[fill=black, circle, inner sep=1.5pt] (A) at (-2, 1) {};
    \node[fill=black, circle, inner sep=1.5pt] (B) at (-2, -1) {};
    \node[fill=gray, circle, inner sep=1.5pt] (U) at (3, 1.7) {};
    \node[fill=gray, circle, inner sep=1.5pt] (V) at (-3, 1.7) {};
    \node[fill=gray, circle, inner sep=1.5pt] (W) at (3, -1.7) {};
    \node[fill=gray, circle, inner sep=1.5pt] (T) at (-3, -1.7) {};
  
    \draw (A) -- (B);
    \draw (A) -- (C);
    \draw (B) -- (D);
    \draw (D) -- (C);
    \draw (E) -- (F);
    \draw (C) -- (E);
    \draw (D) -- (F);
    \draw (E) -- (U);
    \draw (A) -- (V);
    \draw (F) -- (W);
    \draw (T) -- (B);

    \node[] (m1) at (-2.4, 0) {\footnotesize $m_2$};
    \node[] (m2) at (-1, -1.25) {\footnotesize $m_1$};
    \node[] (m3) at (-1, 1.25) {\footnotesize $m_3$};
    \node[] (m4) at (-0.4, 0) {\footnotesize $m_4$};
    \node[] (m5) at (2.45, 0) {\footnotesize $m_7$};
    \node[] (m6) at (1, -1.25) {\footnotesize $m_6$};
    \node[] (m7) at (1, 1.25) {\footnotesize $m_5$};
    \node[] (P1) at (2.6, 1.8) {\footnotesize $p_3$};
    \node[] (P1) at (-2.6, 1.8) {\footnotesize $p_2$};
    \node[] (P1) at (2.6, -1.8) {\footnotesize $p_4$};
    \node[] (P1) at (-2.6, -1.8) {\footnotesize $p_1$};
     \node[] (N) at (0, -2.5) {\footnotesize $G_1$};
\end{tikzpicture}
\end{minipage}
\begin{minipage}{0.3\textwidth}
\centering
\begin{tikzpicture}[scale=0.6]
    \node[fill=black, circle, inner sep=1.5pt] (C) at (0, 1) {};
    \node[fill=black, circle, inner sep=1.5pt] (D) at (0, -1) {};
        \node[fill=black, circle, inner sep=1.5pt] (E) at (2, 1) {};
    \node[fill=black, circle, inner sep=1.5pt] (F) at (2, -1) {};
        \node[fill=black, circle, inner sep=1.5pt] (A) at (-2, 1) {};
    \node[fill=black, circle, inner sep=1.5pt] (B) at (-2, -1) {};
    \node[fill=gray, circle, inner sep=1.5pt] (U) at (3, 1.7) {};
    \node[fill=gray, circle, inner sep=1.5pt] (V) at (-3, 1.7) {};
    \node[fill=gray, circle, inner sep=1.5pt] (W) at (3, -1.7) {};
    \node[fill=gray, circle, inner sep=1.5pt] (T) at (-3, -1.7) {};

        \node[] (P) at (1, 0) {};
    \node[] (Q) at (1, 0) {};
  
    \draw (A) -- (B);
    \draw (A) -- (C);
    \draw (B) -- (D);
    \draw (C) -- (F);
    \draw (P) -- (D);
    \draw (E) -- (Q);
    \draw (C) -- (E);
    \draw (D) -- (F);
    \draw (E) -- (U);
    \draw (A) -- (V);
    \draw (F) -- (W);
    \draw (T) -- (B);

    \node[] (m4) at (0.2, 0.25) {\footnotesize $m_4$};
    \node[] (m5) at (1.85, 0.25) {\footnotesize $m_7$};
    \node[] (m1) at (-2.4, 0) {\footnotesize $m_2$};
    \node[] (m2) at (-1, -1.25) {\footnotesize $m_1$};
    \node[] (m3) at (-1, 1.25) {\footnotesize $m_3$};
    \node[] (m6) at (1, -1.25) {\footnotesize $m_6$};
    \node[] (m7) at (1, 1.25) {\footnotesize $m_5$};
    \node[] (P1) at (2.6, 1.8) {\footnotesize $p_3$};
    \node[] (P1) at (-2.6, 1.8) {\footnotesize $p_2$};
    \node[] (P1) at (2.6, -1.8) {\footnotesize $p_4$};
    \node[] (P1) at (-2.6, -1.8) {\footnotesize $p_1$};
     \node[] (N) at (0, -2.5) {\footnotesize $G_2$};
\end{tikzpicture}
\end{minipage}
\begin{minipage}{0.3\textwidth}
\centering
\begin{tikzpicture}[scale=0.6]
    \node[fill=black, circle, inner sep=1.5pt] (C) at (2, 0) {};
    \node[fill=black, circle, inner sep=1.5pt] (D) at (0, -1) {};
        \node[fill=black, circle, inner sep=1.5pt] (E) at (2, 1) {};
    \node[fill=black, circle, inner sep=1.5pt] (F) at (2, -1) {};
        \node[fill=black, circle, inner sep=1.5pt] (A) at (-2, 1) {};
    \node[fill=black, circle, inner sep=1.5pt] (B) at (-2, -1) {};
    \node[fill=gray, circle, inner sep=1.5pt] (U) at (3, 1.7) {};
    \node[fill=gray, circle, inner sep=1.5pt] (V) at (-3, 1.7) {};
    \node[fill=gray, circle, inner sep=1.5pt] (W) at (3, -1.7) {};
    \node[fill=gray, circle, inner sep=1.5pt] (T) at (-3, -1.7) {};
  
    \draw (A) -- (B);
    \draw (A) -- (E);
    \draw (B) -- (D);
    \draw (D) -- (C);
    \draw (E) -- (F);
    \draw (C) -- (E);
    \draw (D) -- (F);
    \draw (E) -- (U);
    \draw (A) -- (V);
    \draw (F) -- (W);
    \draw (T) -- (B);
 
      \node[] (m1) at (-2.4, 0) {\footnotesize $m_2$};
    \node[] (m2) at (-1, -1.25) {\footnotesize $m_1$};
    \node[] (m3) at (0, 1.25) {\footnotesize $m_3$};
    \node[] (m4) at (0.6, -0.2) {\footnotesize $m_4$};
    \node[] (m5) at (2.45, 0.5) {\footnotesize $m_7$};
    \node[] (m6) at (1, -1.25) {\footnotesize $m_6$};
    \node[] (m7) at (2.45, -0.5) {\footnotesize $m_5$};
    \node[] (P1) at (2.6, 1.8) {\footnotesize $p_3$};
    \node[] (P1) at (-2.6, 1.8) {\footnotesize $p_2$};
    \node[] (P1) at (2.6, -1.8) {\footnotesize $p_4$};
    \node[] (P1) at (-2.6, -1.8) {\footnotesize $p_1$};
     \node[] (N) at (0, -2.5) {\footnotesize $G_3$};
\end{tikzpicture}
\end{minipage}
\end{aligned}
\end{equation}

Convergence of these Feynman integrals in the massless case was studied recently in
parts B, C and D of
\cite[Section IV]{CKN}. 
All three graphs have $\ell = 2$ loops, $N=4$ external edges, and $n=7$ internal edges.
 The graphs have the same parameter space
of dimension $\binom{N}{2} + n = 13$.
Following \cite[Example 1]{MizeraTelen}, we consider $m_1, \dots ,m_7, \, M_1 := k_{11}, \,  M_2 := k_{22}, \, M_3 := k_{33}, \,  M_4 := k_{44}, \, s := M_1+ M_2 + k_{12}, \, t := M_2+ M_3 + k_{23} $ as the basis of the parameter space.

We compare the copositive  cones $\mathcal{C}_{G_1}, \, \mathcal{C}_{G_2}, \, \mathcal{C}_{G_3}$ inside the shared
parameter space~$\RR^{13}$. This uses the labeling in~\eqref{Eq:BoxBeetle}. All three Symanzik polynomials are cubic in $7$ variables. The next table shows the numbers of monomials and the $f$-vector 
for each Newton polytope:
\begin{center}
\begin{tabular}{c|c|c|c}
     & double box $G_1$ & nonplanar double box  $G_2$  & beetle $G_3$  \\
    \hline
    \# monomials & 63 & 66 & 59 \\
    \hline
    $f$-vector & (30,90,121,92,41,10) & (32,96,128,96,42,10) & 
    (28,84,115,90,41,10)  \\
\end{tabular}
\end{center}
Finding a semialgebraic description for the copositive cones of the graphs in~\eqref{Eq:BoxBeetle} 
is difficult. We will not attempt this here.
Instead, we explore how these cones are related to each other.
To simplify the discussion, we assume that the masses of the particles are equal, that is, 
we set $m := m_1 = \dots = m_7\,$ 
and $\,M:=M_1 = \dots = M_4$. Now the copositive cones live in $\mathbb{R}^4$.

Suppose that $Q$ is a face of  $\Newt(\mathcal{F}_z)$, and denote $\mathcal{F}_{z}|_Q$ the polynomial obtained from $\mathcal{F}_{z}$ by collecting the monomials whose exponent vector is contained in $Q$.
If $\mathcal{F}_{z}|_Q$ is not copositive then $\mathcal{F}_{z}$ is not copositive either. Thus, 
we obtain outer approximations of $\mathcal{C}_G$ by intersecting the
copositive cones of $\mathcal{F}_{z}|_Q$ for some faces $Q$.
If $\mathcal{F}_z|_Q$ is maximally sparse (i.e.~all monomials correspond to vertices of $Q$),
 then $\mathcal{F}_z|_Q$ is copositive if and only if all  coefficients are positive. 
  The edges $Q $ for which $\mathcal{F}_z|_Q$ is not maximally sparse correspond to bubble diagrams
  (Example \ref{ex:bubble}).
  From this we obtain the outer approximations
     $\,\mathcal{C}_{G_1} \subseteq \mathcal{E}_1$, $\mathcal{C}_{G_2} \subseteq \mathcal{E}_1$
     and $\mathcal{C}_{G_3} \subseteq \mathcal{E}_2$ with the cones
$$  \mathcal{E}_1 \,:=\, \{4m - M \geq 0 \} \,\cap \,  \{  4m - s \geq 0 \}
\quad {\rm and} \quad    \mathcal{E}_2 \,:= \,\mathcal{E}_1 \cap  \{4m - t \geq 0 \}.
$$

For each two-dimensional face $Q$, if $\mathcal{F}_{z}|_Q$ is not maximally sparse, then it either equals
\begin{equation}
\label{Eq:TwoQuadraticEdge}
(x_k+ x_u)\big(m_ix_i^2+ (m_i+m_j-p)x_ix_j +m_jx_j^2\big) 
\,\,\, \hbox{ \begin{footnotesize} for some $\, i,j,k,u \in [7],  p \in \{M_1,\dots,M_4,s,t\} $, \end{footnotesize}}
\end{equation}
 or it is the second Symanzik polynomial for one of the following three graphs:
\begin{equation}
\label{Eq:ThreeSmallGraphs}
\begin{aligned}
\begin{minipage}{0.3\textwidth}
\centering
\begin{tikzpicture}[scale=0.3]
    \node[fill=black, circle, inner sep=1.5pt]
    (A) at (2.5, 0) {};
     \node[fill=black, circle, inner sep=1.5pt] (B) at (-2.5, 0) {};
     \draw[bend left=70] (A) to (B) node[midway, above=-45 pt] {};
     \draw[bend right=70] (A) to (B) node[midway, above=45 pt] {};
     \draw (A) to (B);
\end{tikzpicture}
\end{minipage}
\begin{minipage}{0.3\textwidth}
\centering
\begin{tikzpicture}[scale=0.3]
    \node[fill=black, circle, inner sep=1.5pt]
    (A) at (-1, 0) {};
     \node[fill=black, circle, inner sep=1.5pt] (B) at (2.5, 2.2) {};
     \node[fill=black, circle, inner sep=1.5pt] (C) at (2.5, -2.2) {};
     \draw (A) to (B);
      \draw (A) to (C);
       \draw (C) to (B);
\end{tikzpicture}
\end{minipage}
\begin{minipage}{0.3\textwidth}
\centering
\begin{tikzpicture}[scale=0.3]
    \node[fill=black, circle, inner sep=1.5pt]
    (A) at (2.5, 0) {};
     \node[fill=black, circle, inner sep=1.5pt] (B) at (-2.5, 0) {};
     \draw[bend left=70] (A) to (B) node[midway, above=-45 pt] {};
     \draw[bend right=70] (A) to (B) node[midway, above=45 pt] {};
      \draw[-] (A) edge[out=30, in=-30, loop,looseness= 50] (A);
\end{tikzpicture}
\end{minipage}
\end{aligned}
\end{equation}
In each case, we derive outer approximations from 
 Theorem~\ref{thm:banana} or Example~\ref{Example:Triangle}.
For $G_1$, the $2$-faces $Q$ corresponding to banana and triangle diagrams give that $\mathcal{C}_{G_1}$ is contained~in
\begin{align*}
\label{Eq:TwoDimFaceBox}
\! & \! \mathcal{D}_1 \,:=\,  \{ 9m - s \geq 0 \} \,\cap \,\{ 9m - M \geq 0 \} \,\cap \,\{ 9m - t \geq 0 \} \,\,\cap \\ 
& \qquad\quad \{ -8Mm^2 + 2M^2m - 2M^2s - ms^2 - 4m^2s + 8Mms + 5m^3 \geq 0  \text{ or } 7m - 2M  - s  \geq 0 \}.
\end{align*}
Similarly, for the non-planar double box $G_2$, we show that the copositive cone is a subset of
\begin{equation}
\label{Eq:TwoDimFaceNonplanarBox}
\begin{aligned}
 \mathcal{D}_2 \,\,:=\,\, \mathcal{D}_1\, \cap\, \{ 9m - 4M + s + t \geq 0 \}.
\end{aligned}
\end{equation}
Using the approximations $\mathcal{E}_1, \, \mathcal{E}_2, \, \mathcal{D}_1, \,\mathcal{D}_2$, we
now compare the three copositive cones for~\eqref{Eq:BoxBeetle}.

\begin{example} \label{ex:ztildez}
We consider parameter vectors $z = (m,M,s,t)$.
 The vector $z = (1,\tfrac{1}{4},-6,-6) $ lies in $ \mathcal{C}_{G_1}\cap \mathcal{C}_{G_3}$
 because the two Symanzik polynomials have positive coefficients. But
 $z$ is not in  $\mathcal{C}_{G_2}$ because it is not in~$\mathcal{D}_2$.
The vector $ \tilde z =   (5,4,-8,22)$ is not in $\mathcal{E}_2$ and hence not in $\mathcal{C}_{G_3}$.
 But $\tilde z \in \mathcal{C}_{G_1} \cap \mathcal{C}_{G_2}$.
  We compute a certificate for this containment in Example~\ref{Ex:BoxesBeetlesManifest}.
        \end{example}

Our discussion underscores the importance of having an explicit description of copositive cones 
for small graphs, like the triangle and  bananas.  These motifs serve as building blocks, 
allowing us to compare larger copositive cones whose full descriptions may be out of reach.

\section{Polynomials With Fixed Support}
\label{sec5}

This section introduces copositive geometry in the general
framework of sparse polynomials developed by
 Gel'fand, Kapranov and Zelevinsky in \cite{GKZ}.
 We show that the copositive cone is bounded by the principal
   ${\mathcal A}$-determinant \cite[Chapter 10]{GKZ}.
 Thereafter, we specialize to the Feynman scenario, where
 the sparse polynomial is $\mathcal{F}_z$, and
 $\mathcal{A}$ comes from  Lemma \ref{lem:NimaSebastian}.

Let $\mathcal{A}$ be any finite subset of $\NN^n$ such that
each element ${\bf a} \in \mathcal{A}$ has the same coordinate sum.
We also assume
that the polytope $P = {\rm conv}(\mathcal{A})$ has dimension $n-1$.
We allow the possibility that $\mathcal{A}$ is a proper subset of
$P \cap \ZZ^n$. Let $\RR[\mathcal{A}]$ be the
real vector space of polynomials in $x = (x_1,\ldots,x_n)$ that
are supported on $\mathcal{A}$. Such a polynomial has the form
\begin{equation}
\label{eq:suchpoly}
  f(x) \,= \, \sum_{{\bf a} \in \mathcal{A}} c_{\bf a} \,x^{\bf a},
\qquad {\rm where} \,\,\,c_{\bf a} \in \RR. \end{equation}
We can regard $f(x)$ as a linear function on the affine
toric variety  $X_\mathcal{A} \subset \RR[\mathcal{A}]$. This toric
variety is a cone, and we identify it with the corresponding
projective toric variety.
We write $X_{\mathcal{A},\geq 0}$ for the closed
semialgebraic subset of all points with nonnegative coordinates.

The {\em copositive cone} is the set of polynomials that are nonnegative
on nonnegative points:
\begin{equation}
\label{eq:CAdef}
 \mathcal{C}_\mathcal{A} \,\, = \,\,\bigl\{
f \in \RR[\mathcal{A}] \,: \, f(u) \geq 0 \quad \hbox{for all} \,\, u \in \RR^n_{\geq 0} \bigr\} \,\,
= \,\, \bigl\{f \in \RR[\mathcal{A}] \,: \, f \geq 0 \,\,{\rm on} \,\, X_{\mathcal{A},\geq 0} \bigr\}. 
\end{equation}
This is a full-dimensional closed convex cone in $\RR[\mathcal{A}]$.
Polynomials in $\mathcal{C}_\mathcal{A}$ are called {\em copositive}.

\begin{lemma} \label{eq:onboundary}
Let $f$ be a polynomial in the copositive cone $\mathcal{C}_\mathcal{A}$.
Then $f$ lies on the boundary of  $\,\mathcal{C}_\mathcal{A}$
if and only if $\,f$ has a zero in the nonnegative part 
$\,X_{\mathcal{A},\geq 0}\, $ of the toric variety $\,X_\mathcal{A}$.
\end{lemma}

\begin{proof}
The nonnegative toric variety $X_{\mathcal{A},\geq 0}$ 
is a compact semialgebraic subset of the simplex
which is formed by the nonnegative part of the
ambient projective space. Every polynomial $f \in \RR[\mathcal{A}]$ 
represents a function on that simplex, and $f$ is copositive
if and only if $f$ is nonnegative on the subset $X_{\mathcal{A},\geq 0}$.
In this case, $f$ is strictly positive on $X_{\mathcal{A},\geq 0}$ if and only if
 some open neighborhood of $f$ is contained in $\mathcal{C}_\mathcal{A}$
 if and only if $f$ is in the interior of $\mathcal{C}_\mathcal{A}$.
\end{proof}

\begin{remark} \label{rmk:dualcone}
The second formula in (\ref{eq:CAdef}) says that
the copositive cone $\mathcal{C}_\mathcal{A}$ is the cone dual to the convex cone 
in $\RR[\mathcal{A}]$ spanned by
the positive toric variety $X_{\mathcal{A},\geq 0}$. In symbols, we have
\begin{equation}
\label{eq:dualcone} \mathcal{C}_\mathcal{A} \,\, = \,\,
{\rm conv}( X_{\mathcal{A},\geq 0})^\vee. 
\end{equation}
\end{remark}

For every face $Q$ of $P$ we write $f|_Q$ for the subsum of all
monomials $c_{\bf a} x^{\bf a}$ where ${\bf a} \in Q$.

\begin{lemma}
\label{Lemma:InteriorCoposCone}
The interior of the copositive cone $\mathcal{C}_\mathcal{A}$ consists of all polynomials
$f \in \RR[\mathcal{A}]$ such that, for each face $Q$ of $P$, the inequality 
$f|_Q(u) > 0$ holds for all  $u $ in the open orthant~$\RR_{> 0}^n$.
\end{lemma}

\begin{proof}
We consider the decomposition of $X_{\mathcal{A}}$ into
torus orbits. There is one orbit for each face $Q$ of $P$,
and this orbit is parametrized by monomials
$x^{\bf a}$ where ${\bf a} $ ranges over $\mathcal{A} \cap Q$.
The nonnegative toric variety $X_{\mathcal{A},\geq 0}$ is the
disjoint union of the sets of strictly positive points in each orbit.
The interior of $\mathcal{C}_\mathcal{A}$ thus consists of polynomials
$f$ that are strictly positive on each such positive torus orbit, and this
is precisely the stated conjunction over all $Q$.
\end{proof}

\begin{theorem} \label{thm:CA}
The boundary of $\mathcal{C}_\mathcal{A}$ is contained in the 
hypersurface  defined by the principal $\mathcal{A}$-determinant $E_\mathcal{A}$, 
which is the product of the discriminants $\Delta_{\mathcal{A} \cap Q}$
over all faces $Q$ of $P$.
\end{theorem}

\begin{proof}
Let $f \in \RR[\mathcal{A}]$ and suppose that $f$ is in the boundary of $\mathcal{C}_\mathcal{A}$.
By Lemma~\ref{eq:onboundary}, there exists a face
   $Q$ of the Newton polytope $P$ such that $f|_Q(u)=0$ for some 
   $u \in \mathbb{R}^n_{>0}$.
  Since $f$ is copositive, $u$ is a local minimum of $f|_Q$.
  This  implies that all $n$ partial derivatives of $f|_Q$ also vanish at the point $u$. 
  Therefore, $f$ lies on the hypersurface defined by $\Delta_{\mathcal{A} \cap Q}$.
\end{proof}

Theorem \ref{thm:CA} says that
the algebraic boundary of $\mathcal{C}_\mathcal{A}$ is a hypersurface contained in
the hypersurface defined by 
the principal $\mathcal{A}$-determinant $E_\mathcal{A}$.
Generally, they are not equal.

\begin{example}
    The algebraic boundary of $\mathcal{C}_\mathcal{A}$ can be strictly contained in the principal $\mathcal{A}$-determinant. This is the case for the square
   $\mathcal{A} = \{ (0,0,2), \, (1,0,1), \, (0,1,1), \, (1,1,0) \}$.
    Indeed, the quadratic form  $\,  f(x) = c_1  x_3^2 + c_2 x_1 x_3 + c_3 x_2 x_3 + c_4 x_1x_2 \,$
    is copositive if and only if all coefficients $c_i$ are nonnegative,
    because  $\mathcal{A}$ is the set of  vertices of $\conv(\mathcal{A})$.
In symbols,
    \begin{align*}
        \mathcal{C}_\mathcal{A} \,= \,\{c \in \mathbb{R}^4 \mid c_1 \geq 0, \, c_2 \geq 0, \, c_3 \geq 0, \, c_4 \geq 0\}.
    \end{align*}
    The algebraic boundary of $\mathcal{C}_\mathcal{A}$ is the
    reducible hypersurface defined by the polynomial $c_1c_2c_3c_4$.
    On the other hand, the principal $\mathcal{A}$-determinant 
    equals $E_\mathcal{A} = c_1c_2c_3c_4(c_1c_4-c_2c_3)$.
\end{example}

We now return to Feynman integrals.
Fix a Feynman graph $G$ with $n$ internal edges.
As before, $\mathcal{T}$ denotes the set of all spanning trees of the graph $G$.
These are the bases of the graphic matroid of $G$. By definition, the support of  $\mathcal{F}_z(x)$ 
equals
$\mathcal{A}_G \,=\, \mathcal{A}^1_G \cup \mathcal{A}^2_G$, where
\begin{align}
\label{Eq:supportFz}
  \mathcal{A}^1_G = \Big\{ \,e_k \,+\!\! \sum_{j\in [n] \backslash T } \!e_j \,\mid \,T \in \mathcal{T} ,\,k \in T \Big\},
\quad 
    \mathcal{A}^2_G\, =\, \Big\{ 2e_q \,+\!\!\! \sum_{j\in [n]\backslash (T \cup \{q\})}\!\!\!
    \!\! e_j \,\mid\, T \in \mathcal{T} ,\; q \not\in T \Big\}.
 \end{align}

\begin{example}[Parachute]
\label{Ex:Parachute}
We discuss this for the following graph with $n=4$ and $\ell=2$:
    \begin{equation}
\label{Eq:Parachute}
\begin{aligned}
\begin{tikzpicture}[scale=0.4]
    \node[] (P4) at (2.7, 3.5) {\small $p_4$};
     \node[] (P3) at (2.7, -3.5) {\small $p_3$};
    \node[] (P1) at (-5.6, 1.5) {\small $p_1$};
     \node[] (P2) at (-5.6, -1.5) {\small $p_2$};
    \node[fill=black, circle, inner sep=1.5pt]
    (A) at (-3, 0) {};    
     \node[fill=black, circle, inner sep=1.5pt] (B) at (0, 1.8) {};
       \node[fill=black, circle, inner sep=1.5pt] (C) at (0, -1.8) {};
     \node[fill=gray, circle, inner sep=1.5pt] (B1) at (2, 3.3) {};
      \node[fill=gray, circle, inner sep=1.5pt]  (B2) at (2, -3.3) {};
    \node[fill=gray, circle, inner sep=1.5pt] (A1) at (-5, 1) {};
      \node[fill=gray, circle, inner sep=1.5pt] (A2) at (-5, -1) {};
     \draw[bend left=30] (B) to (C);
     \draw[bend right=30] (B) to (C);
      \draw (A) -- (B);
          \node[] (m1) at (-1.8, 1.3) {\footnotesize $m_1$};
          \node[] (m2) at (-1.8, -1.3) {\footnotesize $m_2$};
             \node[] (m3) at (-1.2, 0) {\footnotesize $m_3$};
          \node[] (m4) at (1.2, 0) {\footnotesize $m_4$};
    \draw (A) -- (C);
       \draw (A) -- (A1);
    \draw (A) -- (A2);
           \draw (B) -- (B1);
    \draw (C) -- (B2);
\end{tikzpicture}
\end{aligned}
\end{equation}
The set of spanning trees equals
$ \,\mathcal{T} = \bigl\{\{2,4\},\;\{2,3\},\;\{1,4\},\;\{1,3\},\;\{1,2\}\bigr\}$. In pictures,
\begin{center}
\begin{minipage}{0.2\textwidth}
\centering
\begin{tikzpicture}[scale=0.4]
    \node[fill=black, circle, inner sep=1.5pt]
    (A) at (-3, 0) {};    
     \node[fill=black, circle, inner sep=1.5pt] (B) at (0, 1.8) {};
       \node[fill=black, circle, inner sep=1.5pt] (C) at (0, -1.8) {};
     \draw[bend left=30] (B) to (C);
          \node[] (m2) at (-1.8, -1.3) {\footnotesize $m_2$};
             \node[] (m4) at (1.2, 0) {\footnotesize $m_4$};
    \draw (A) -- (C);
\end{tikzpicture}
\end{minipage}
\hspace{-10pt}
\begin{minipage}{0.2\textwidth}
\centering
\begin{tikzpicture}[scale=0.4]
    \node[fill=black, circle, inner sep=1.5pt]
    (A) at (-3, 0) {};    
     \node[fill=black, circle, inner sep=1.5pt] (B) at (0, 1.8) {};
       \node[fill=black, circle, inner sep=1.5pt] (C) at (0, -1.8) {};
     \draw[bend right=30] (B) to (C);
          \node[] (m2) at (-1.8, -1.3) {\footnotesize $m_2$};
             \node[] (m3) at (-1.2, 0) {\footnotesize $m_3$};
    \draw (A) -- (C);
\end{tikzpicture}
\end{minipage}
\hspace{-10pt}
\begin{minipage}{0.2\textwidth}
\centering
\begin{tikzpicture}[scale=0.4]
    \node[fill=black, circle, inner sep=1.5pt]
    (A) at (-3, 0) {};    
     \node[fill=black, circle, inner sep=1.5pt] (B) at (0, 1.8) {};
       \node[fill=black, circle, inner sep=1.5pt] (C) at (0, -1.8) {};
     \draw[bend left=30] (B) to (C);
      \draw (A) -- (B);
          \node[] (m1) at (-1.8, 1.3) {\footnotesize $m_1$};
             \node[] (m4) at (1.2, 0) {\footnotesize $m_4$};
\end{tikzpicture}
\end{minipage}
\hspace{-10pt}
\begin{minipage}{0.2\textwidth}
\centering
\begin{tikzpicture}[scale=0.4]
    \node[fill=black, circle, inner sep=1.5pt]
    (A) at (-3, 0) {};    
     \node[fill=black, circle, inner sep=1.5pt] (B) at (0, 1.8) {};
       \node[fill=black, circle, inner sep=1.5pt] (C) at (0, -1.8) {};
     \draw[bend right=30] (B) to (C);
      \draw (A) -- (B);
          \node[] (m1) at (-1.8, 1.3) {\footnotesize $m_1$};
            \node[] (m3) at (-1.2, 0) {\footnotesize $m_3$};
\end{tikzpicture}
\end{minipage}
\hspace{-10pt}
\begin{minipage}{0.2\textwidth}
\centering
\begin{tikzpicture}[scale=0.4]
    \node[fill=black, circle, inner sep=1.5pt]
    (A) at (-3, 0) {};    
     \node[fill=black, circle, inner sep=1.5pt] (B) at (0, 1.8) {};
       \node[fill=black, circle, inner sep=1.5pt] (C) at (0, -1.8) {};
      \draw (A) -- (B);
          \node[] (m1) at (-1.8, 1.3) {\footnotesize $m_1$};
          \node[] (m2) at (-1.8, -1.3) {\footnotesize $m_2$};
    \draw (A) -- (C);
\end{tikzpicture}
\end{minipage}
\end{center}

The support set $\mathcal{A} = \mathcal{A}_G$ consists of $14=4+10$ lattice points in $\NN^4$. It is the union of
\begin{equation}
 \label{Eq:ExponentsParachute}
  \begin{matrix}
  \mathcal{A}_G^1 &=& \{ (1,1,1,0), \; (1,1,0,1), \; (1,0,1,1), \; (0,1,1,1) \} \qquad {\rm and} \smallskip \\
    \mathcal{A}_G^2 &=&  \{ (2,0,1,0),  (1,0,2,0), (2,0,0,1),  (1,0,0,2),
    (0,2,1,0),  \\ & & \quad\,   (0,1,2,0),  (0,2,0,1),  (0,1,0,2),  (0,0,1,2),  (0,0,2,1) \}.
\end{matrix}
\end{equation}

\begin{figure}[h]
\centering
\includegraphics[scale=0.4]{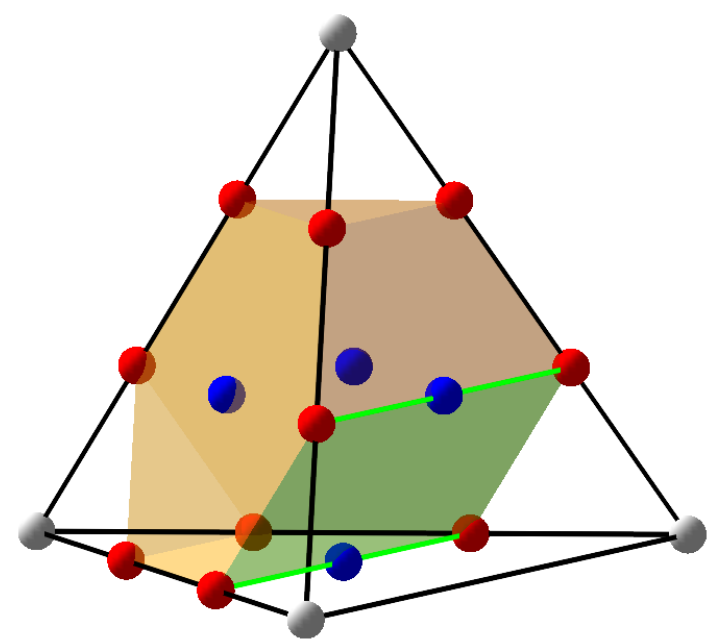}
\caption{The Newton polytope for the parachute graph $G$,
with $\mathcal{A}_G^1$ in blue and $\mathcal{A}_G^2$ in red.
}
\label{FIG:Parachute}
\end{figure}

\noindent This configuration is shown in Figure \ref{FIG:Parachute}.
The general polynomial with support $\mathcal{A}_G$ equals
$$  \begin{matrix}
f(x) & = & c_1 x_1^2x_3 + c_2 x_1x_3^2  + c_3 x_1^2x_4 + c_4 x_1x_4^2 +
c_5 x_2^2x_3+ c_6 x_2x_3^2 +  c_7 x_2^2x_4+ c_8 x_2x_4^2  \\ & & + \, c_9 x_3^2x_4 + c_{10} x_3x_4^2 +
c_{11} x_1x_2x_3 + c_{12} x_1x_2x_4 + c_{13} x_1x_3x_4+ c_{14} x_2x_3x_4.
\end{matrix}
$$
The algebraic boundary  of the  copositive cone $\mathcal{C}_{\mathcal{A}} \subset \RR^{14}$
is in the hypersurface defined by the principal ${\mathcal{A}}$-determinant $E_\mathcal{A}$.
To study this, we use that the toric threefold $X_\mathcal{A}$
is smooth and has degree $18$ in $\PP^{13}$. This implies that
$E_\mathcal{A}$ has degree $4 \cdot 18 = 72$.
The irreducible factors correspond to faces of $P = {\rm conv}(\mathcal{A})$.
The largest factor, of degree $24$, is  the $\mathcal{A}$-discriminant.
The two hexagonal facets contribute discriminants of degree $12$.
The green rectangle facet contributes the resultant of two binary quadrics,
which has degree $4$. The two trapezoid facets contribute cubic factors
$c_2^2 c_5+c_1 c_6^2-c_2 c_6 c_{11}$ and
$c_4^2 c_7+c_3 c_8^2-c_4 c_8 c_{12}$.
The two long edges of the green rectangle have discriminants of degree $2$,
and the ten vertices give
$c_1 c_2 c_3 c_4 c_5 c_6 c_7 c_8 c_9 c_{10}$.
The degrees of all factors add up to
$ 24 + (2\cdot 12 + 4 +  2\cdot 3) + 2 \cdot 2 + 10 \cdot 1 = 72$.

The specialization of $f(x)$ arising from the Feynman graph depends on only $7$ parameters:
$$  \! \mathcal{F}_z \,=\,a x_1 x_2 (x_3+x_4)+b x_2 x_3 x_4+c x_1 x_3 x_4
+((x_1+x_2) (x_3+x_4)+x_3 x_4)(m_1 x_1+m_2 x_2+m_3 x_3+m_4 x_4). $$
The discriminant of $\mathcal{F}_z$ is much smaller than the $\mathcal{A}$-discriminant.
It has $64$ terms of degree~$6$:
\begin{equation}
\label{eq:landauparachute} \begin{matrix}
L \,\,= \,\, 16m_1^2 m_2^2 m_3^2 -32 m_1^2 m_2^2 m_3 m_4-32 m_1^2 m_2 m_3^2 m_4
 - \,\cdots\,
+8 m_2 m_3 m_4 a c^2 \qquad \phantom{dodo} \\ \quad +\,2 m_2 m_3 a^2 c^2  +2 m_2 m_4 a^2 c^2
-4 m_3 m_4 a^2 c^2+2 m_1 m_2 b^2 c^2+a^2 b^2 c^2-2 m_2 a b c^3+m_2^2 c^4.
\end{matrix}
\end{equation}
This is the Landau discriminant of the parachute $G$.
We find that $z = (m_1,m_2,m_3,m_4,a,b,c)$ is in the interior of the copositive cone $\mathcal{C}_G$
provided
 $L < 0$ and $   a + m_1 + m_2 + 2(m_1m_2)^{1/2}  > 0$,
$$ \begin{matrix}
&     b + m_2 + m_3 + m_4 + 2(m_2m_3)^{1/2} + 2(m_2m_4)^{1/2} + 2 (m_3m_4)^{1/2}  > 0, \\
{\rm and} &     c + m_1 + m_3 + m_4 + 2 (m_1 m_3)^{1/2} + 2(m_1 m_4)^{1/2} + 2 (m_3 m_4)^{1/2}  > 0.
\end{matrix}          $$

For a numerical instance let $ m_1 \!=\! 5, m_2 \!=\! 7, m_3 \!=\! 7, m_4 \!=\! 2, 
 a\! =\! -16, b \!=\! -36, c \!=\! -31$.
The minimum of $\mathcal{F}_z(x)$ on the tetrahedron
$\{x \geq 0 : x_1 +x_2 + x_3 + x_4 = 1\}$ equals
$0.01389365$.
This is attained at 
$(x_1,x_2,x_3,x_4) = (0.211294, 0.1870297, 0.208148, 0.393528)$.
For a certificate, note that all $1823248$ monomials in
 $\mathcal{F}_z(x) \cdot (x_1+x_2+x_3+x_4)^{217}$ have positive coefficients.
\end{example}

We now return to the general setting, where $\mathcal{A} = \mathcal{A}_G$.
The principal $\mathcal{A}$-determinant $E_\mathcal{A}$
describes the copositive cone $\mathcal{C}_\mathcal{A}$ algebraically.
In our application to physics, we are interested in the 
smaller cone $\mathcal{C}_G$, which is
the intersection with the kinematic subspace in
  $\RR[\mathcal{A}]$. In other words,
  $\mathcal{C}_G$ is obtained from $\mathcal{C}_\mathcal{A}$
by replacing (\ref{eq:suchpoly}) with the polynomial  in (\ref{Eq:F}).
For $\ell \geq 2$, this restriction usually leads to a considerable
simplification in the copositive cone.
We saw this in Example~\ref{Ex:Parachute}.
The situation for $\ell=1$ is different, thanks to
Theorem~\ref{thm:universality}.

We conclude this section by transferring Remark \ref{rmk:dualcone}
and Theorem \ref{thm:CA} from the larger cone $\mathcal{C}_\mathcal{A}$
to the smaller cone $\mathcal{C}_G$. 
Fix a Feynman graph $G$, and write $M$ for the number of kinematic parameters in $z$.
Their inclusion into the   space of coefficients $c_{\bf a}$ is dual to a linear
projection $\pi: \RR[\mathcal{A}] \rightarrow \RR^M$. The image of 
 the toric variety $X_\mathcal{A}$ under the linear projection $\pi$ is
 a rational variety $\pi(X_\mathcal{A})$ inside $\RR^M$. This is
 the cone over a projective variety in $\PP^{M-1}$.
 
 \begin{example}[Parachute]
 When equating $f(x)$ with $\mathcal{F}_z(x)$ in Example \ref{Ex:Parachute},
 each of the $14$ coefficients $c_i$ is
an $\NN$-linear combination of $a,b,c,m_1,m_2,m_3,m_4$.
This defines the linear map $\pi: \RR^{14} \rightarrow \RR^7$.
The image of the toric threefold $X_\mathcal{A}$ is the
threefold in $\PP^6$ parametrized~by
$$  \begin{matrix}
a= x_1x_2 x_3+x_1 x_2 x_4\,,\,\,\,
b = x_2 x_3 x_4\,,\,\,\,
c =  x_1 x_3 x_4\,, \\
m_1 = x_1^2 x_3 \!+\!x_1^2 x_4\!+\!x_1 x_2 x_3\!+\!x_1 x_2 x_4\!+\!x_1 x_3 x_4, \,
m_2 = x_1 x_2 x_3\!+\!x_1 x_2 x_4 \!+\! x_2^2 x_3\!+\! x_2^2 x_4\!+\!x_2 x_3 x_4,\\
m_3 = x_1 x_3^2 + x_1 x_3 x_4 + x_2 x_3^2 + x_2 x_3 x_4 + x_3^2 x_4, \,
m_4 = x_1 x_3 x_4 + x_1 x_4^2 + x_2 x_3 x_4 + x_2 x_4^2 + x_3 x_4^2.
\end{matrix}
$$
The threefold $\pi(X_\mathcal{A})$ has degree $7$, and its prime ideal equals
$\langle b m_1-c m_2, a m_1+a m_2+c m_2-m_1 m_2, a b+a c+b c-c m_2,
     c m_2 m_3+c m_2 m_4-a m_3 m_4, b m_2 m_3+b m_2 m_4+a m_3m_4+b m_3 m_4-m_2 m_3 m_4,
     c m_1 m_3+c m_1 m_4+a m_3 m_4+c m_3 m_4-m_1 m_3 m_4  \rangle   $.
     The variety  projectively dual to $\pi(X_\mathcal{A})$   is the sextic hypersurface 
defined by the Landau discriminant $L$ displayed in (\ref{eq:landauparachute}).
 \end{example}

We write $\pi(X_{\mathcal{A},\geq 0})$ for the nonnegative part of the 
variety $\pi(X_\mathcal{A})$, now viewed as an affine cone in $\RR^M$. This is the closure of all points 
given by some positive parameter point $x \in \RR^n_{> 0}$.

\begin{corollary} The copositive cone $\mathcal{C}_G$ of a Feynman graph $G$ is the convex cone 
which is dual to the
convex cone spanned by the positive variety  $\pi(X_{\mathcal{A},\geq 0})$. In symbols,  we have
$$ \quad \mathcal{C}_G \,\,\, = \,\,\, {\rm conv} \bigl( \pi(X_{{\mathcal A},\geq 0}) \bigr)^\vee \quad 
\subset \quad \RR^M. $$
\end{corollary}

We now turn to the {\em principal Landau determinant}  
which was introduced by Fevola, Mizera and Telen in \cite[Definition 3.5]{FMT}.
This is the squarefree polynomial ${\rm PLD}_G(z)$ which takes over
the role of the principal $\mathcal{A}$-determinant after
the toric variety $X_\mathcal{A}$ has been replaced by its image under
the projection $\pi$.
Namely, the principal Landau determinant ${\rm PLD}_G(z)$ vanishes whenever the hypersurface 
given by  the second Symanzik polynomial $\mathcal{F}_z$, or one of its facial restrictions
$\mathcal{F}_z|_Q$, is singular at some point
$x \in (\mathbb{C}^*)^n$. Using the same argument as in the proof of 
Theorem \ref{thm:CA}, we can derive the following result
for any Feynman graph $G$.

\begin{corollary} \label{cor:CG}
The boundary of the copositive cone $\mathcal{C}_G$ is contained in the 
hypersurface  defined by the principal Landau determinant ${\rm PLD}_G$.
Hence,  the algebraic boundary of $\mathcal{C}_G$ is the product of 
some of the Landau discriminants that appear in the factorization of  ${\rm PLD}_G$.
\end{corollary}

We conclude that the algorithms in \cite[Section 5]{FMT} will be
important ingredients in future methods for computing semialgebraic
descriptions of copositive cones
for Feynman integrals.

\section{Pólya's Method for Feynman Graphs}
\label{sec6}
\label{Sec:PolyaSymanzik}

In this section we
apply Pólya's classical representation   \cite{Polya, PowersReznick} to the second Symanzik polynomial~\eqref{Eq:F}.
We assume throughout that $G$ is a Feynman graph without bridges, that is, $G$ cannot be disconnected by removing a single edge.
We establish the following result:

\begin{theorem}
\label{Thm:PolyaSymanzik}
A parameter vector $z$ lies in the interior of the copositive cone $\mathcal{C}_G$ if and only if $(x_1+\dots+x_n)^N\mathcal{F}_{z}$ has only positive coefficients,
for some $N \in \mathbb{N}$, and $m_1,\dots ,m_n > 0$.
\end{theorem}

Our proof is based on work of Castle, Powers, and Reznick in \cite{CastlePowersReznick}.
To recall this, we introduce some notation.  We write  $\Delta$ for the standard $(n-1)$-simplex
 in $\RR^n$. Every $I \subseteq [n]$ determines a \emph{face} $\Delta_I = \{ x \in \Delta \mid x_i = 0 \text{ for } i \in I\}$. For a polynomial
 $f = \sum_{{\bf a} \in \mathcal{A}}  c_{{\bf a}} x^{\bf a}$, we~set
 \[ \mathcal{A}^+\, :=\, \{ {\bf a} \in \mathbb{N}^n \mid c_{\bf a} > 0 \}
\quad  {\rm and} \quad  \mathcal{A}^- \,:=\, \{ {\bf a} \in \mathbb{N}^n \mid c_{\bf a} < 0 \} .\]
We now define a partial order on $\mathbb{N}^n$ that depends on $I$.
 For ${\bf a}, {\bf b} \in \mathbb{N}^n$, we set ${\bf a} \preceq_I {\bf b}$ if ${\bf a}_i \leq {\bf b}_i$ for $i \in I$. If one of these inequalities is strict, then we write ${\bf a} \prec_I {\bf b}$. An element ${\bf a} \in \mathcal{A}^+$ is \emph{minimal for $I $} if there is no ${\bf a}' \in \mathcal{A}^+$ such that ${\bf a}' \prec_I {\bf a}$. 
 If this holds then we define 
\begin{align}
\label{Eq:falphaI}
 \Gamma_{I,{\bf a}} \,\,:= \,\,\{\,{\bf b } \in \mathcal{A} \;|\; {\bf b}_i ={\bf a}_i \text{ for } i \in I\}.
\end{align}

We next state {\em P\'olya's Theorem with Zeros}. This  was established in
\cite[Theorem 2]{CastlePowersReznick}.

\begin{theorem}\label{Thm:PolyaWithZeros}
      Let $f = \sum_{{\bf a} \in \mathcal{A}}  c_{{\bf a}} x^{\bf a}$ be a copositive polynomial
      such that
        $V(f) \cap \Delta$ is a union of faces of $\Delta$. 
        There exists $N \in \mathbb{N}$ such that $(x_1+\dots+x_n)^N f(x)$ has positive~coefficients
         if and only if,  for every face $\Delta_I $ in the hypersurface $ V(f)$, 
         the following two conditions hold: \vspace{-0.12cm}
      \begin{itemize}
          \item[(1)] For each ${\bf b}$ in the negative support $ \mathcal{A}^-$, there exists ${\bf a} \in \mathcal{A}^+$ such that ${\bf a} \preceq_I {\bf b}$. \vspace{-0.12cm}
          \item[(2)] If ${\bf a} \in \mathcal{A}^+$ is minimal for $I$, then the polynomial $f|_{\Gamma_{I,{\bf a}}}$ is positive on $\mathbb{R}^n_{>0}$.
      \end{itemize}
\end{theorem}

The hypothesis $\Delta_I \subseteq V(f)$ 
means that every ${\bf a} \in   \mathcal{A}$ satisfies ${\bf a}_i > 0$ for some $i \in I$. If $(x_1+\dots+x_n)^N f(x)$ has positive coefficients, then $f$ lies in the interior of the copositive cone~$\mathcal{C}_\mathcal{A}$. The converse does not hold for arbitrary polynomials (see Example~\ref{Ex:PolyaFails}). However, 
it is true for Symanzik polynomials, as we will show by applying Theorem~\ref{Thm:PolyaWithZeros} to $\mathcal{F}_z$.
We begin by assuming $m_1 > 0,\ldots, m_n > 0$.
This ensures that $\mathcal{A}^2_G \subseteq \mathcal{A}^+_G$ and
$\mathcal{A}^-_G \subseteq \mathcal{A}^1_G$, where $\mathcal{A}^+_G, \, \mathcal{A}^-_G$ is the positive and negative support of $\mathcal{F}_z$ respectively, with $\mathcal{A}^1_G,\, \mathcal{A}^2_G$ as in~\eqref{Eq:supportFz}.

\begin{lemma}
\label{Lemma:Condition1}
For every $I \subsetneq [n]$ and every  ${\bf b} \in  \mathcal{A}_G^1$, there exists  ${\bf a} \in \mathcal{A}_G^2$ such that ${\bf a} \preceq_I {\bf b}$.
    In particular, the second Symanzik polynomial $\mathcal{F}_z$ satisfies condition (1) in Theorem~\ref{Thm:PolyaWithZeros}.
\end{lemma}

\begin{proof}
We write ${\bf b} = e_k + \sum_{j\in [n] \setminus T} e_j$ where
 $T \in \mathcal{T}$ and $k \in T$. 
  If $q \in [n] \backslash (T \cup I)$, then ${\bf a} := 2 e_{q} + \sum_{j \in [n]\setminus(T\cup\{q\})} e_j 
$ lies in $\mathcal{A}_G^2$ and it satisfies ${\bf a} \preceq_I {\bf b}$. 
If $T \cup I = [n]$, then we pick an element  $q \in T \backslash I$. 
  Since $q$ is not a bridge in $G$, there exists $T_2 \in \mathcal{T}$ such that $q \notin T_2$. From the basis exchange axiom for matroids,
   it follows that there exists $j' \in [n] \setminus T$ such that $T_3 := \big(T \cup \{ j' \} \big) \backslash \{q\} \in \mathcal{T}$. Thus, for ${\bf a} := 2 e_{q} + \sum_{j \in [n] \setminus(T_3 \cup \{q\})} e_j \in \mathcal{A}_G^2$ we have ${\bf a} \preceq_I {\bf b}$.\end{proof}

In Lemma~\ref{Lemma:MinimalGammas}, we show that the minimal sets in~\eqref{Eq:falphaI} correspond to faces of $\conv(\mathcal{A}_G)$. 
To warm up for this,
 we compute some minimal exponent vectors of $\mathcal{A}_G$ for the parachute.

\begin{example}
\label{Ex:ParachuteGammas}
    Fix the parachute diagram from Example~\ref{Ex:Parachute} and $I = \{3,4\}$. 
    We seek the vectors in \eqref{Eq:ExponentsParachute} that are minimal for $\preceq_I$.
   An exponent vector in $\mathcal{A}_G$
   is minimal if and only if its last two coordinates are either $(1,0)$ or $(0,1)$. Thus, we have two minimal sets 
\begin{align}
\label{Eq:GammeEdgesParachute}
    \Gamma^{(1)}_{I,{\bf a}} =  \big\{ (2,0,1,0), (1,1,1,0), (0,2,1,0) \big\}, \;\;  \Gamma^{(2)}_{I,{\bf a}} = \big\{ (2,0,0,1), \; (1,1,0,1),\; (0,2,0,1) \big\}.
\end{align}
These correspond to two green edges of the 
Feynman polytope ${\rm conv}(\mathcal{A}_G)$ in Figure~\ref{FIG:Parachute}. 
\end{example}

We now show that the same property holds for any Feynman graph $G$ without bridges.

\begin{lemma}
\label{Lemma:MinimalGammas}
Consider any subset  $I \subsetneq [n]$ and any point ${\bf a} \in \mathcal{A}_G^+$ 
that is minimal for $I$. Then $\,
\Gamma_{I,{\bf a}} = \mathcal{A}_G \cap Q\,$ for some face $Q$ of
the Feynman polytope $\conv(\mathcal{A}_G)$.
\end{lemma}

\begin{proof}
We  fix the subset $I$.
    If ${\bf a} \in \mathcal{A}_G^1$, then Lemma~\ref{Lemma:Condition1} tells us that there exists ${\bf a}' \in \mathcal{A}_G^2$, which is minimal for $I$ and satisfies ${\bf a}'_i = {\bf a}_i$ for $i \in I$. Thus, 
    we can assume ${\bf a} \in \mathcal{A}_G^2$. 
    There exists a spanning tree $T \in \mathcal{T}$ and $q \in [n]\backslash T$ such that ${\bf a} = 2e_q + \sum_{j \in [n]\setminus(T  \cup \{q\})} e_j$. 

    In the first part of the proof, we show by contradiction that $q \notin I$. Assume  $q \in I$. If there exists $p \in [n] \backslash (T\cup I)$, then $2e_p + \sum_{j \in [n] \setminus (T \cup \{p\})} e_j \prec_I {\bf a}$.
    This is a contradiction to ${\bf a}$ being minimal. Hence $T \cup I = [n]$.
    By the basis exchange axiom applied to $p \in T \setminus I$, there exists $j' \in [n] \backslash T$ such that $T_2 := (T \cup \{j'\}) \backslash \{p\} \in \mathcal{T}$. By construction we have $2e_p + \sum_{j \in [n] \setminus (T_2 \cup \{p\})} e_j \prec_I {\bf a}$. This is again a contradiction, and we conclude that $q \notin I$.

Our goal is to identify a face $Q$ of  the polytope $\conv(\mathcal{A}_G)$. To this end, we now set
$$ v \,\,:=\,\, \sum_{j \in I \setminus T}\! e_j + \sum_{i \in I \cap T} (\#(I\backslash T)+1) \,e_i.$$
Let  $Q$  be the face with inner normal vector $v$.
 By construction, $\langle v, {\bf b} \rangle = \#(I \backslash T)$ for ${\bf b} \in \Gamma_{I,{\bf a}}$. The intersection of $\mathcal{A}_G$ with $Q$ is the set $\Gamma_{I,{\bf a}}$ if and only if
  $\langle v, {\bf b} \rangle > \#(I \backslash T)$ for ${\bf b} \in \mathcal{A}_G \backslash \Gamma_{I,{\bf a}}$. To prove this, we show that for every ${\bf b} \in \mathcal{A}_G$ such that $\langle v, {\bf b} \rangle \leq \#(I\backslash T)$, we have  ${\bf b} \in \Gamma_{I,{\bf a}}$. By the choice of $v$, the inequality $\langle v, {\bf b} \rangle \leq \#(I\backslash T)$ ensures that ${\bf b}_i = 0$ for $i \in I \cap T$.

First assume that ${\bf b} \in \mathcal{A}_G^2$. We choose a spanning three
 $T' \in \mathcal{T}$ and $p \in [n] \backslash T'$ such that 
 ${\bf b} = 2e_p + \sum_{j \in [n] \backslash (T' \cup \{p\})} e_j$. 
Since ${\bf b}_i = 0$ for $i \in I \cap T$, we have $I \backslash T' \subseteq I \backslash T$. 
If equality holds then  ${\bf b} \in \Gamma_{I,{\bf a}}$. 
We consider the case  $I \backslash T' \subsetneq I \backslash T$. 
Since $T',T \in \mathcal{T}$, we cannot have $[n] \backslash T' \subseteq I$. 
Thus, there exists $p' \in [n] \backslash (T' \cup I)$, 
and we have  $2e_{p'} + \sum_{j \in [n] \setminus (T' \cup \{p'\})} e_j \prec_I {\bf a}$. 
But, recall that  ${\bf a}$ is minimal for~$I$.
We conclude that the case $I \backslash T' \subsetneq I \backslash T$ cannot happen.

For the second case,  consider any point ${\bf b} \in \mathcal{A}_G^1$ such that 
$\langle v, {\bf b} \rangle \leq \#(I \backslash T)$. By Lemma~\ref{Lemma:Condition1}, there exists ${\bf b}' \in \mathcal{A}_G^2$ with ${\bf b}' \preceq_I {\bf b}$. This implies  $\langle v, {\bf b}' \rangle \leq \langle v, {\bf b} \rangle \leq \#(I\backslash T)$. Since ${\bf b}' \in \mathcal{A}_G^2$, the argument in the previous
paragraph shows that ${\bf b}' \in \Gamma_{I,{\bf a}}$. Thus, we have $\#(I\backslash T) =  \langle v, {\bf b}' \rangle \leq \langle v, {\bf b} \rangle \leq \#(I\backslash T)$, and therefore ${\bf b} \in \Gamma_{I,{\bf a}}$.
This completes the proof.
\end{proof}

 Using Lemma~\ref{Lemma:Condition1} and \ref{Lemma:MinimalGammas}, we prove that Pólya's 
  certificate works whenever the second Symanzik polynomial of a Feynman graph $G$
   lies in the interior of the copositive cone $\mathcal{C}_G$.
 
\begin{proof}[Proof of Theorem \ref{Thm:PolyaSymanzik}]
The copositive cone $\mathcal{C}_G$ is the intersection of 
the cone $\mathcal{C}_{\mathcal{A}_G}$ from Section~\ref{sec5} with the kinematic subspace $\mathcal{K} \cong \mathbb{R}^{\binom{N}{2}+n}$ described in Section~\ref{sec2}. For fixed $k_{ij}$ and large enough 
masses $m_e$, the polynomial $\mathcal{F}_z$ in \eqref{Eq:F} has only positive coefficients. Thus, $\mathcal{K} \cap \interior(\mathcal{C}_{\mathcal{A}_G})\neq \emptyset$, which, in turn, implies that $\relint(\mathcal{C}_G) = \interior( \mathcal{C}_{\mathcal{A}_G}) \cap \mathcal{K}$.

Suppose $(x_1+\cdots+x_n)^N\cdot \mathcal{F}_{z}$ has only positive coefficients for some
$N \in \mathbb{N}$, and assume
$m_1 > 0, \dots, m_n > 0$. This latter hypothesis implies that $\Newt(\mathcal{F}_z) = \conv(\mathcal{A}_G)$. Consider any face $Q$ of this polytope. The initial form of the above
product in direction $Q$ equals $(\sum_{i \in J} x_i)^N \cdot \mathcal{F}_z|_Q(x)$
for some $J \subset [n]$.
Since all coefficients of this product are positive, we have
 $\, \mathcal{F}_z|_Q(x) > 0$ for all $x \in \mathbb{R}^n_{>0}$. 
This is equivalent to $z \in \relint(\mathcal{C}_G)$ by Lemma \ref{Lemma:InteriorCoposCone}.

    Conversely, assume that $\mathcal{F}_z|_Q(x) > 0$ for all faces $Q \subseteq \conv(\mathcal{A}_G)$ and
    all points $x \in \mathbb{R}^n_{>0}$.  This implies that $m_1>0,\dots,m_n > 0$, $\mathcal{F}_z(x) \geq 0$ for all $ x \in \Delta$, and $V(\mathcal{F}_z) \cap \Delta$ is a union of faces of $\Delta$. By Lemma~\ref{Lemma:Condition1}, $\mathcal{F}_z$ satisfies condition (1) in Theorem~\ref{Thm:PolyaWithZeros}. Using our assumption and Lemma~\ref{Lemma:MinimalGammas}, we conclude that condition (2) in Theorem~\ref{Thm:PolyaWithZeros} is also satisfied. Thus, P\'olya's Theorem with Zeros
(Theorem~\ref{Thm:PolyaWithZeros}) shows that the desired integer  $N \in \mathbb{N}$ exists.
 \end{proof}

We conclude this section by showing that condition (1) in Theorem~\ref{Thm:PolyaWithZeros} can fail for 
arbitrary polynomials $f(x)$. Polynomials from Feynman graphs are special.
The point is that
the minimal sets in~\eqref{Eq:falphaI} do not always correspond to faces of the Newton polytope.

\begin{example}
\label{Ex:PolyaFails}
Consider the ternary quintic
    \[f(x_1,x_2,x_3)\,\, =\,\,  x_1^3x_2x_3 -x_1^2x_2^2x_3 + x_1x_2^3x_3 + x_2^5
    \,\,=\,\, x_1x_2x_3\big( (x_1 - x_2)^2 + x_1x_2\big) + x_2^5.   \]
The Newton polytope $    \Newt(f)$ is a triangle. For each of the
 faces $Q$  of   $ \Newt(f)$, one checks that
$f|_Q(x) > 0$ for all $x \in \mathbb{R}^3_{>0}$.
Nevertheless, for all $N \in \mathbb{N}$, the product
 $(x_1+x_2+x_3)^Nf(x)$ has both positive and negative coefficients.
       To see this, we apply the only-direction in  Theorem~\ref{Thm:PolyaWithZeros} 
       to the edge $\Delta_{\{1,2\}}$, which is contained in $V(f)$. The positive and negative support of $f$ are $\mathcal{A}^+ = \big\{ (3,1,1), \; (1,3,1), \; (0,5,0)\big\}$ and $\mathcal{A}^- = \big\{ (2,2,1)\big\}$.
    For ${\bf b} = (2,2,1)$ there is no ${\bf a} \in \mathcal{A}^+$ with ${\bf a} \preceq_{\{1,2\}} {\bf b}$. Condition (1) in Theorem~\ref{Thm:PolyaWithZeros} is violated and therefore  $(x_1+x_2+x_3)^Nf(x)$ does not have only positive coefficients for any $N \in \mathbb{N}$.
\end{example}

\begin{example}
    The minimal sets in \eqref{Eq:falphaI} do not always correspond to faces of the Newton polytope. 
    To illustrate this, we consider ternary forms of degree $10$ which have support sets
    \[\mathcal{A}^+ = \big\{ (1,2,7),(2,5,3),(5,4,1),(4,4,2)\big\}
    \quad {\rm and} \quad \mathcal{A}^- = \big\{ (2,3,5)\big\}.\]
    Every element in the positive support $\mathcal{A}^+$ is minimal for $I =\{1,3\}$. For ${\bf a} = (4,4,2)$, 
    the minimal set is $\Gamma_{I,{\bf a}} = \{{\bf a}\}$, and this is contained in the interior of 
    the triangle  $\conv(\mathcal{A}^+)$.
\end{example}

\section{Computing Certificates}
\label{sec7}

Given a Feynman graph $G$ and a choice of kinematic parameters $z$,
we wish to decide whether $\mathcal{F}_z(x)$ is copositive.
The answer should be made manifest with a certificate.
If $z$ is in the interior of $\mathcal{C}_G$ then this can be
certified with Pólya's method, as shown in Theorem \ref{Thm:PolyaSymanzik}.
If $z$ lies outside the cone $\mathcal{C}_G$
then the certificate consists of a positive point $u \in \RR^n_{\geq 0}$
such that $\mathcal{F}_z(u) < 0$.
The boundary of $\mathcal{C}_G$ is the decision boundary.
The probability for random data $z$ to lie in this boundary is zero.
We thus ignore this case for our discussion in this section.

We present an algorithm for 
making that decision and for computing the certificates.
We also offer a proof-of-concept implementation in 
the programming language {\tt Julia}.
Our software is called \texttt{CopositiveFeynman.jl}.
 The implementation rests on the 
computer algebra system \texttt{OSCAR.jl}~\cite{OscarBook},
and it calls the packages \texttt{Landau.jl} for Landau discriminants
 \cite{MizeraTelen}
and \texttt{HomotopyContinuation.jl} for numerical algebraic geometry \cite{BreidingTimme}.
The code is posted~at
 $$ \hbox{\url{https://mathrepo.mis.mpg.de/CopositiveFeynman}.} $$

The input for our software is the graph $G$ and a vector $z$ of kinematic parameters.
The decision whether $z$ lies in $\mathcal{C}_G$ or not amounts to solving the following 
optimization problem:
\begin{equation}
\label{eq:optimization}
{\rm Minimize} \,\,\mathcal{F}_z(u) \,\,\hbox{subject to} \,\,\, u \in \Delta. 
\end{equation}
Here $\Delta$ denotes the $(n-1)$-simplex $\{ u \in \RR^n_{\geq 0} : u_1 + u_2 + \cdots + u_n = 1\}$.
We compute the minimum in (\ref{eq:optimization}) algebraically.
Namely we replace $\Delta$ with
the Feynman polytope $P $, and we solve
the critical equations on each face $Q$ of $P$.
The details for this will be described below.

If the objective function value in (\ref{eq:optimization})
is positive then we invoke P\'olya's method
to certify that $\mathcal{F}_z$ lies in the interior of~$\mathcal{C}_G$. 
By Theorem~\ref{Thm:PolyaSymanzik}, there exists a positive integer $N$ such that
$(x_1+\dots+x_n)^N \mathcal{F}_z(x)$ has only positive coefficients. 
Our function \texttt{find\_Polya\_exponent()} finds the smallest integer $N$  with this property.
An upper bound for that $N$ was given by
Castle, Powers and Reznick in \cite[Theorem 3]{CastlePowersReznick}.
 Their bound depends on the degree, the size of the coefficients, and the minimal values the polynomial attains
 on $\Delta$. In a nutshell, the closer to zero the minimal value in
 (\ref{eq:optimization}) happens to be, the larger will be the bound on $N$.

We illustrate the growth of $N$ near the boundary of $\mathcal{C}_G$ for the banana diagrams
in Section~\ref{sec4}. The Symanzik polynomial $\mathcal{F}_z$ has $n+1$ kinematic parameters $z = (m_1, \dots, m_n, s)$. Setting $m_1 = 1, \dots, m_n = 1$, the parameter $z$ lies in the copositive cone if and only if $s/n^2 \leq 1$, by Theorem~\ref{thm:banana}. 
Figure \ref{FIG:PolyaExponent} features $n=2,3,4$. It shows
 several choices of $s/n^2$ along with the smallest $N$ such that $(x_1+\dots+x_n)^N \mathcal{F}_z(x)$ has only positive coefficients.

\begin{figure}[h]
\centering
\includegraphics[scale=0.4]{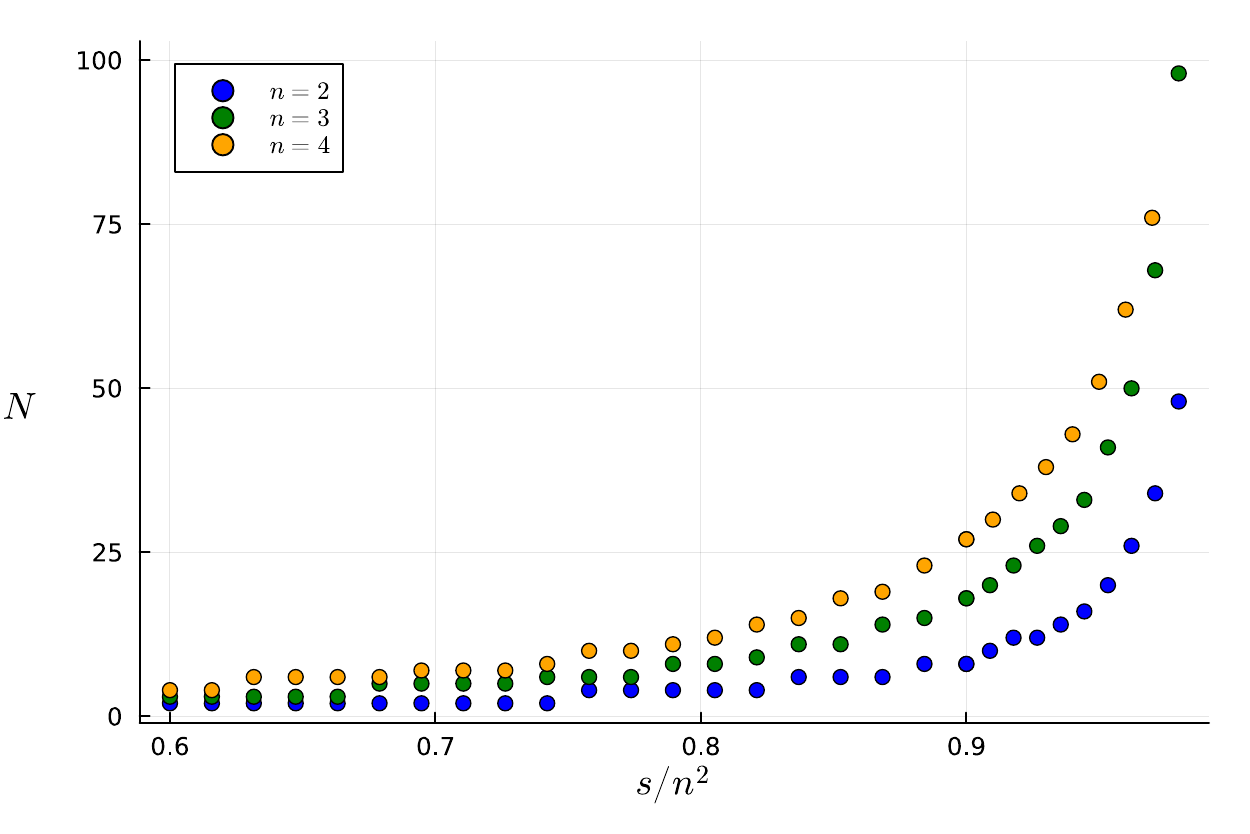}
\caption{{\small Illustration of P\'olya's Theorem for banana diagrams with $n$ internal edges and $z=(1,\dots,1,s)$.
The smallest certifying exponent $N$ is a function of $s/n^2$. It blows up when $s/n^2 \rightarrow 1$.}}
\label{FIG:PolyaExponent}
\end{figure}

Our approach to the optimization problem (\ref{eq:optimization}) rests on the
affine critical equations
\begin{equation}
\label{Eq:AffineLandau}
        \frac{\partial\mathcal{F}_{z}}{\partial x_1}(x_1,\dots,x_{n-1},1) \,=\, \cdots \,=\, 
        \frac{\partial{\mathcal{F}_z}}{\partial x_{n-1}}(x_1,\dots,x_{n-1},1)  \,\,= \,\,0.
\end{equation}
These are called \emph{Landau equations} in physics.
We solve the equations (\ref{Eq:AffineLandau}) using the software
 \texttt{HomotopyContinuation.jl} \cite{BreidingTimme}.
 The number of solutions is finite, possibly after perturbing the coefficients.
 We extract all solutions $u$ with positive real coordinates, and we
 evaluate $\mathcal{F}_{z}$ at these solutions.
 If $\mathcal{F}_z(u) < 0$ for some $u$ then we are done:
  $\mathcal{F}_z$ is manifestly not copositive.
 
 Otherwise, we examine the facets $Q$ of 
 the Feynman polytope $P = {\rm Newt}(\mathcal{F}_z)$.
 For each facet $Q$, we form the restriction
$\mathcal{F}_{z}|_Q$, and we consider the affine Landau equations
for $\mathcal{F}_{z}|_Q$.
Again, if one critical point $u \in \RR^n_{> 0}$ satisfies
$\mathcal{F}_{z}|_Q(u) < 0$ then
we are done: a certificate for $z \not\in \mathcal{C}_G$ has been found.
Otherwise, we proceed to facets of $Q$. In this manner we examine
all faces of $P$. 
This approach is justified by Lemma \ref{Lemma:InteriorCoposCone},
and it verifies that $\mathcal{F}_z$ is {\em completely non-vanishing}, in the
language of \cite{Borinsky}.
At this point we know that $\mathcal{F}_z$ is in the interior of the copositive cone.
But we still need a certificate.
For that,  our software invokes the
function \texttt{find\_Polya\_exponent()}. This now
computes a P\'olya certificate for the copositivity of $\mathcal{F}_z$.

\begin{example}
\label{Ex:ComputingParachute}
We revisit the parachute diagram (Example~\ref{Ex:Parachute}) and show how to use the package \texttt{CopositiveFeynman.jl}. To compute the Symanzik polynomials, we rely on the
package \texttt{Landau.jl} 
by Mizera and Telen \cite{MizeraTelen}. Here, a graph is represented by a list of \texttt{edges}, where each edge is 
given by its pair of vertices, and a list of \texttt{nodes}, which specify the vertices to which the external edges are attached. For the parachute diagram~\eqref{Eq:Parachute}, we write:

\medskip
\begin{verbatim}
using CopositiveFeynman
edges = [[1,2],[1,3],[2,3],[2,3]];
nodes = [1,1,3,2];
Fz, U, x, k, mm = getF(edges, nodes);
Fz,s,t,M,m = substitute4legs(Fz,k,mm);
\end{verbatim}
The function \texttt{substitute4legs()} takes $m_1, \, m_2, \, m_3, \, m_4, \, M_1 := k_{11}, \,  M_2 := k_{22}, \, M_3 := k_{33}, \,  M_4 := k_{44}, \, s := M_1+ M_2 + k_{12}, \, t := M_2+ M_3 + k_{23} $ as the basis of the parameter space (cf. Section~\ref{sec4}). 
We choose numerical values for these parameters so that $\mathcal{F}_z$ is copositive.
\begin{verbatim}
F = subs(Fz,m[1]=>1,m[2]=>1,m[3]=>1,m[4]=>1,s=>3.9,t=>1,M[3]=>1,M[4]=>1);
find_Polya_exponent(F);
\end{verbatim}
The code returns $N=37$. This means that
 $(x_1+\dots+x_n)^N\mathcal{F}_{z}(x)$ has positive coefficients, certifying that $\mathcal{F}_z$ is copositive. We 
 next modify the parameters so that $\mathcal{F}_z$ is not copositive:
\begin{verbatim}
F = subs(Fz,m[1]=>1,m[2]=>1,m[3]=>1,m[4]=>1,s=>4.1,t=>1,M[3]=>1,M[4]=>1);
preclude_copositivity(F,edges);
\end{verbatim}
This returns the edge $Q_1 = \conv( (2,0,1,0),(0,2,1,0))$ of the Feynman polytope in Figure~\ref{FIG:Parachute},
and also the evaluation
   $\mathcal{F}_z|_{Q_1}(1.05,1,1,1) = -0.1025$.  Hence $\mathcal{F}_z$ is manifestly not copositive. 

\medskip

The polytope in Figure~\ref{FIG:Parachute} has seven facets and $15$ edges.
Not all faces $Q$ are needed for solving the affine critical equations~\eqref{Eq:AffineLandau}. 
For instance, for the green rectangle, we observe
\begin{equation}
\label{Eq:ParachuteFactors}
\begin{aligned}
 \mathcal{F}_z|_Q &\,=\, (x_3 + x_4) \bigl(\,m_1 x_1^2 + (m_1 + m_2 -s)x_1x_2 + m_2 x_2^2 \,\bigr).
\end{aligned}
\end{equation}
If this fails to be copositive then so does 
$\mathcal{F}_z|_{Q_1}(x)$, where $Q_1 = \conv( (2,0,1,0),(0,2,1,0))$. 
The same holds for the
 edge $Q_2 = \conv( (2,0,0,1),(0,2,0,1))$.
 Thus, instead of minimizing each of $\mathcal{F}_z|_{Q}, \, \mathcal{F}_z|_{Q_1}, \, \mathcal{F}_z|_{Q_2}$ individually, it 
 suffices to do so for $\mathcal{F}_z|_{Q_1}$. 
Note that in \eqref{Eq:ParachuteFactors}, $x_3+x_4$ is the first Symanzik polynomial of the subgraph $\gamma$ with edges $3$ and $4$, while $m_1 x_1^2 + (m_1 \!+\! m_2 \!-\!s)x_1x_2 + m_2 x_2^2$ is the second Symanzik polynomial of the contraction $G/\gamma$.
\end{example}

The observation above is true for any Feynman graph.
Facets $Q$ of $\Newt(\mathcal{F}_z)$ correspond to one-vertex irreducible subgraphs $\gamma$ of $G$. These are subgraphs that cannot be disconnected by removing a single vertex. 
From \cite[Theorem 2.7]{Brown}, we learn that
\begin{align}
\label{Eq:ExpOnFacets}
 \mathcal{F}_z|_{Q}(x) \,\,= \,\,\mathcal{U}_\gamma(x) \cdot \mathcal{F}_{G/\gamma}(x) .
\end{align}

In this formula, $\mathcal{U}_\gamma$ is the first Symanzik polynomial of the subgraph $\gamma$,
and $\mathcal{F}_{G/\gamma}$ is the second Symanzik polynomial of the contraction $G/\gamma$.
Thus, when we examine the facet $Q$ corresponding to $\gamma$, it suffices to 
solve the affine Landau equations for  $G/\gamma$.
Indeed, $\mathcal{F}_z|_{Q}(x)$ fails to be strictly copositive
if and only if $\mathcal{F}_{G/\gamma}$ does. We now go down in dimension.
  If  $Q_1$ is a facet of $Q$, then
  $\mathcal{F}_z|_{Q_1}(x) = (\mathcal{U}_{\gamma})|_{Q_U}(x) \cdot (\mathcal{F}_{G/\gamma})|_{Q_F}(x)$ where $Q_U$ and $Q_F$ are faces of $\Newt(\mathcal{U}_\gamma)$ and $\Newt(\mathcal{F}_{G/\gamma})$ respectively, with the same inner normal vector as $Q_1$. Thus, to find points where $\mathcal{F}_z|_{Q_1}$ attains negative values,
  it suffices to explore the faces of $\Newt(\mathcal{F}_{G/\gamma})$.

The facets of $\Newt(\mathcal{F}_{G/\gamma})$ are given by  one-vertex irreducible subgraphs $\gamma_2$ of $G/\gamma$. For these, we solve the affine Landau equations for  $(G/\gamma)/\gamma_2$. We continue this reduction until the contracted graph 
has one vertex and no edges. We summarize this procedure in Algorithm~\ref{algo}.
\begin{algorithm}[H]
\caption{\texttt{preclude\_copositivity}}
\label{algo}
\begin{algorithmic}[1]
\Require  $(G,\mathcal{F}_z)$, a Feynman graph and its second Symanzik polynomial
\Ensure \True\, if $\mathcal{F}_z$ is not copositive, \False\, otherwise
\State $X \leftarrow$ positive solutions of \eqref{Eq:AffineLandau} for $\mathcal{F}_z$
\If{$\mathcal{F}_z(x) < 0$ for some $x \in X$}
    \Return \True
\EndIf

\State \texttt{subgraphs} $\leftarrow$  $(\gamma,\mathcal{F}_{G/\gamma})$ for all one-vertex irreducible subgraphs $\gamma$ of $G$ 

\For{ $(\gamma,\mathcal{F}_{G/\gamma}) \in $ \texttt{subgraphs} }
    \If{\texttt{preclude\_copositivity}$(\gamma,\mathcal{F}_{G/\gamma})$}
    \Return \True
    \EndIf
\EndFor
\Return \False
\end{algorithmic}
\end{algorithm}
Our software  \texttt{CopositiveFeynman.jl} offers a test implementation of
Algorithm~\ref{algo}. We conclude by running it on the Feynman graphs in \eqref{Eq:BoxBeetle},
for parameters $z$ we saw in Section~\ref{sec4}.

\begin{example}
\label{Ex:BoxesBeetlesManifest}
Fix the double box $G_1$, nonplanar double box $G_2$ and beetle graph $G_3$,
with four kinematic parameters $z =(m,M,s,t) $. We run our code on all
three graphs for $z = (1,\tfrac{1}{4},-6,-6) $  and for
$ \tilde z =   (5,4,-8,22)$. These are the parameters from
Example \ref{ex:ztildez}.

    For $z$, the smallest Pólya exponent for $G_1$ and $G_3$ is $N = 0$.
    Indeed, for these two graphs, $\mathcal{F}_z$ has only positive coefficients. For $G_2$, the function \texttt{preclude\_copositivity()} gives that
\begin{align*}
    \mathcal{F}_z|_{Q}(x) \,\,= \,\,x_2x_6^2 + x_2^2x_6 + x_2^2x_5  + x_6^2x_5  + x_2x_5^2  + x_5^2x_6 - 10x_2x_5x_6 \,\,\approx \,\,-13.7588
\end{align*}
for $x = (1, 2.5351, 1, 1, 2.5352, 1, 1)$, and it exhibits the
relevant two-dimensional face $Q$ of the Feynman polytope.
For $\tilde z$, the second Symanzik polynomial of $G_3$ is not copositive since
\begin{align*}
    \mathcal{F}_{\tilde{z}}|_{Q}(x) \,\,=\,\, 5x_2^2x_4  - 12x_2x_4x_7  + 5x_4x_7^2 \,\,\approx \,\,-2.2
\end{align*}
for $x = (1, 1.2, 1, 1, 1, 1, 1 )$ and the face $Q=\conv \{(0,2,0,1,0,0,1), (0,1,0,1,0,0,2) \}$.
For $G_1$ and $G_2$,  the product $(x_1 + \dots +x_7)\mathcal{F}_{\tilde{z}}(x)$ has only positive coefficients, certifying copositivity.
\end{example}

\bigskip \bigskip 

\noindent {\bf Acknowledgement}:
Supported by the ERC \begin{small} (UNIVERSE PLUS, 101118787). \end{small}
$\!\!$ \begin{scriptsize}Views~and~opinions expressed
are however those of the authors only and do not necessarily reflect those of the European Union or the European
Research Council Executive Agency. Neither the European Union nor the granting authority
can be held responsible for them.
\end{scriptsize}

\bigskip

\noindent {\bf Conflict of Interest Statement}:
On behalf of all authors, the corresponding author states that there is no conflict of interest.

\bigskip \medskip

\noindent
\footnotesize {\bf Authors' addresses:}
\smallskip

\noindent Bernd Sturmfels, MPI-MiS Leipzig \hfill {\tt bernd@mis.mpg.de}

\noindent M\'at\'e L. Telek, MPI-MiS Leipzig \hfill {\tt  mate.telek@mis.mpg.de}


\bigskip 

\begin{thebibliography}{1}
\begin{small}
\setlength{\itemsep}{-0.4mm}

\bibitem{AHM}
N.~Arkani-Hamed, A.~Hillman and  S.~Mizera:
{\em Feynman polytopes and the tropical geometry of
UV and IR divergences}, Physical Review D {\bf 105} (2022) 125013.

\bibitem{BKT}
M.~Bodirsky, M.~Kummer and A.~Thom:
{\em Spectrahedral shadows and completely positive maps on real closed fields},
Journal of the European Mathematical Society (2025).

\bibitem{BomzeDur}
I.~Bomze, M.~Dür, E.~de~Klerk, C.~Roos, A.~Quist, and T.~Terlaky:
{\em On copositive programming and standard quadratic optimization problems},
J.~Global~Optimization {\bf 18} (2000) 301--320.

\bibitem{Borinsky}
M.~Borinsky: {\em Tropical Monte Carlo quadrature for Feynman integrals},
Ann.~Inst.~Henri Poincaré Comb.~Phys.~Interact.
{\bf 10} (2023)  635--685.
  

\bibitem{BorinskyMunchTellander}
M.~Borinsky, H.~J. Munch, and F.~Tellander:
{\em Tropical Feynman integration in the Minkowski regime},
Computer Physics Communications~{\bf 292} (2023) 108874.

\bibitem{BreidingTimme} P. Breiding and S. Timme:
{\em HomotopyContinuation.jl: A Package for Homotopy Continuation in Julia},
    in: Mathematical Software – ICMS 2018. ICMS 2018
    (eds.~J.~Davenport, M.~Kauers, G.~Labahn and J.~Urban),
    (2018), 458–465.

\bibitem{Brown}
F. Brown: {\em Feynman amplitudes, coaction principle, and cosmic Galois group},
Communications in Number Theory and Physics  {\bf 11} (2017) 453--556.

\bibitem{CFS}
V.~Calvo Cortes, H.~Frost and B.~Sturmfels:
{\em Kinematic stratifications},
{\tt arXiv:2503.09571}.

\bibitem{CastlePowersReznick}
M.~Castle, V.~Powers, and B.~Reznick:
{\em P\'olya's theorem with zeros},
Journal of Symbolic Computation~{\bf 46} (2011) 1039--1048.

\bibitem{OscarBook}
W.~Decker, C.~Eder, C.~Fieker, M.~Horn,
and M.~Joswig: {\em The {C}omputer {A}lgebra {S}ystem {OSCAR}: {A}lgorithms and {E}xamples},  Springer, 2025.

\bibitem{CKN} L.~de la Cruz, D.~Kosower and P.~Novichkov:
{\em Finite integrals from Feynman polytopes},
Physical Review D {\bf 111} (2025) 105013.

\bibitem{deKlerkPasechnik}
E.~de~Klerk and D.~V. Pasechnik:
{\em Approximation of the stability number of a graph via copositive  programming},
 {\em SIAM J.~Optimization}~{\bf 12} (2002) 875--892.

\bibitem{Dur} M.~D{\"u}r:
{\em Copositive programming -- a survey},
    in: Recent Advances in Optimization and its Applications in Engineering
    (eds.~M.~Diehl, F.~Glineur, E.~Jarlebring, W.Michiels),
    2010, 3--20.

\bibitem{FMT}
C.~Fevola, S.~Mizera, and S.~Telen:
{\em Principal Landau determinants},
Computer Physics Communications {\bf 303} (2024), art.~109278.

\bibitem{GKZ}
I.~Gel'fand, M.~Kapranov, and A.~Zelevinsky: {\em Discriminants, Resultants, 
and Multidimensional Determinants},  Birkh\"auser, Boston, 1994.

\bibitem{Hadeler}
K.~P.~Hadeler: {\em On copositive matrices},
Linear Algebra and its Applications~{\bf 49} (1983) 79--89.

\bibitem{HennRaman}
J.~Henn and P.~Raman: {\em Positivity properties of scattering amplitudes}, Journal of High Energy Physics
{\bf 2025} (2025) art.~150.

\bibitem{LaurentVargas}
M.~Laurent and L.~Vargas:
{\em Exactness of Parrilo’s conic approximations for copositive matrices
  and associated low order bounds for the stability number of a graph},
Mathematics of Operations Research {\bf 48} (2022)  1017--1043.

\bibitem{MizeraTelen}
S.~Mizera and S.~Telen:
{\em Landau discriminants},
 J.~High Energ.~Phys.~(2022), no.~8, art.~200.
 

\bibitem{Motzkin}
T.~S.~Motzkin: {\em Copositive quadratic forms},
Report 1818, National Bureau of Standards, 1952.

\bibitem{Nie:Book}
J.~Nie: {\em Moment and Polynomial Optimization},
Society for Industrial and Applied Math., 2023.


\bibitem{Parrilo} P.~Parrilo:
 {\em Structured Semideﬁnite Programs and Semi-algebraic Geometry
  Methods in Robustness and Optimization},
Ph.D. thesis, California Institute of Technology, Pasadena, 2000.

\bibitem{Polya}
G.~Pólya: {\em Über positive Darstellung von Polynomen},
Naturforsch.~Ges.~Z\"urich {\bf 73} (1928) 141--145;
 in {\em Collected Papers} {\bf 7} (1974), MIT Press, 309--313.
 
\bibitem{PowersReznick}
V.~Powers and B.~Reznick:
{\em A new bound for Pólya's theorem with applications to polynomials  positive on polyhedra},
Journal of Pure and Applied Algebra {\bf 164} (2001) 221--229.

\bibitem{RST}
K.~Ranestad, B.~Sturmfels and S.~Telen:
{\em What is positive geometry?},
Le Matematiche {\bf 80} (2025) 3--16.


\bibitem{SchweighoferVargas}
M.~Schweighofer and L.~F.~Vargas:
{\em Sum-of-squares certificates for copositivity via test states},
 SIAM~J.~Appl.~Algebra~Geom. {\bf 8}~(2022) 797-–820.

\bibitem{Vargas}
L.~Vargas: {\em Sum-of-Squares Representations for Copositive Matrices and
  Independent Sets in Graphs},
Ph.D. thesis, Tilburg University, 2023.

\bibitem{Weinzierl} S.~Weinzierl: {\em Feynman Integrals}, Springer Verlag, 2022.
\end{small}
\end{thebibliography}
\end{document}